\documentclass{amsart}
\usepackage[english]{babel}
\usepackage{graphicx}
\usepackage[mathscr]{eucal}
\usepackage{latexsym, amssymb, amsthm, amsmath}
\usepackage{xspace, pict2e}
\usepackage[svgnames]{xcolor}
\usepackage{ifthen}
\usepackage{tikz}
\usepackage{subfig}
\usetikzlibrary{shapes,arrows}

\usepackage{cite}
\usepackage[all]{xy}
\usepackage{maplestd2e}

\mathchardef\hyph="2D

\newcommand{\CC}{\mathbb{C}}

\newcommand{\PP}{\mathbb{P}}
\newcommand{\Oh}{\mathcal{O}}

\newcommand{\polyring}{\mathbb{C}[x_0,x_1,x_2,x_3,x_4]}
\newcommand{\doublequotient}{/\!/}

\newcommand{\maxsemistabfam}{M_{\leq 0}(\lambda)}
\newcommand{\maxunstabfam}{M_{<0}(\lambda)}

\newtheorem{lemma}{Lemma}
\newtheorem{theorem}[lemma]{Theorem}
\newtheorem{definition}[lemma]{Definition}

\newtheorem{prop}[lemma]{Proposition}

\newtheorem{remark}[lemma]{Remark}

\title{The GIT Compactification of Quintic Threefolds}
\author{Chirag Lakhani}
\address{Department of Mathematics \& Computer Science \\
         High Point University \\
         High Point, NC, 27262\\
         USA}
\email{clakhani@highpoint.edu}

\begin{document}

\begin{abstract}
In this article, we study the geometric invariant theory (GIT) compactification of quintic threefolds.  We study singularities, which arise in non-stable quintic threefolds, thus giving a partial description of the stable locus.  We also give an explicit description of the boundary components and stratification of the GIT compactification.
\end{abstract}

\maketitle 

\thispagestyle{plain}

%%%%%%%%%%%%%%%%%%%%%%%%%%%%%%%%%%%%%%%%%%%%%%%%%%%%%%%%%%%%%%%%%%%%%%
\medskip

\section{Introduction}
\label{sec:intro}
\medskip

\par
Quintic threefolds are some of the simplest examples of Calabi-Yau varieties.  Physicists have given Calabi-Yau varieties a great deal of attention, in the last 30 years, because they give the right geometric conditions for some superstring compactifications \cite{wittenstrominger}.  In mirror symmetry, in particular, the Kahler moduli space and complex structure moduli space of Calabi-Yau varieties are important objects of study.  The purpose of this paper is to describe the complex structure moduli space of quintic threefolds using geometric invariant theory (GIT).
\par
GIT constructions of moduli spaces of projective varieties are automatically projective, therefore have a natural compactification.  The Hilbert-Mumford criterion provides a numerical tool, which is useful when constructing moduli spaces using GIT.  Despite having this tool, it is still difficult to construct moduli spaces in dimension 2 or higher.  There are a some cases where such moduli spaces have been constructed, such as degree 2 and degree 4 $K3$ surfaces by Shah \cite{shahdeg2,shahdeg4}, cubic threefolds by Allcock and Yokoyama \cite{allcock,yokoyama}, and cubic fourfolds by Laza \cite{laza}.   
\par
A quintic threefold is the zero set of a homogeneous degree $5$ form $f \in \polyring$.  The space $H^0(\PP^4,\Oh_{\PP^4}(1)) \cong \CC^5$ represents the set of degree $1$ forms in $\PP^4$.  The parameter space of quintic threefolds is then represented by $V:=\PP(Sym^5(\CC^5))$, which is the projectivization of the space of coefficients of quintic forms $f$.   $\CC^5$  and $Sym^5(\CC^5)$ both have natural $SL(5,\CC)$-actions.  Two threefolds are equivalent if one form can be transformed into another by an $SL(5,\CC)$-action.  In order to construct the moduli space using GIT, the stable and semistable quintic threefolds must be identified.  A quintic threefold $X$ is semistable if there is a $SL(5,\CC)$-invariant function on $\PP(V)$ where $X$ does not vanish.  A semistable quintic threefold $X$ is stable if its $SL(5,\CC)$-orbit, in the space of semistable quintic threefolds, is closed and the isotropy group of $X$ is finite.  The space of semistable quintic threefolds is denoted $\PP(V)^{ss}$ and the space of stable quintic threefolds is denoted $\PP(V)^{s}$.  The $SL(5,\CC)$-orbits of threefolds in $\PP(V)^{s}$ are closed, so the quotient $\PP(V)^{s} \doublequotient SL(5,\CC)$ forms an orbit space.  The addition of semistable quintic threefolds compactifies the moduli space by making it a projective variety i.e. $\PP(V)^{ss} \doublequotient SL(5,\CC)$ is projective.  Two semistable quintic threefolds $X,Y \in \PP(V)^{ss} \setminus \PP(V)^{s}$ map to the same point in $\PP(V)^{ss} \doublequotient SL(5,\CC)$ if their closures satisfy 
\begin{equation}
\overline{SL(5,\CC)\cdot X} \cap \overline{SL(5,\CC)\cdot Y} \neq \emptyset.
\label{eq:orbitclosure}
\end{equation}

This establishes an orbit-closure relationship for semistable quintic threefolds where $X \sim Y$ if they satisfy the property \ref{eq:orbitclosure}.  Furthermore, all threefolds in the same orbit-closure equivalence class map to the same point in $\PP(V)^{ss} \doublequotient SL(5,\CC)$.  Every orbit-closure equivalence class in $\PP(V)^{ss} \setminus \PP(V)^{s}$ has a unique closed orbit representative called the \textit{minimal orbit}.  The boundary components of $\PP(V)^{ss} \doublequotient SL(5,\CC) \setminus \PP(V)^{s} \doublequotient SL(5,\CC)$ are represented precisely by these minimal orbits.  

\begin{remark}  We will follow the terminology in GIT \cite{mumford}. Unstable will mean not semistable, non-stable will mean not stable, and strictly semistable will mean semistable but not stable.
\end{remark}

The main results of this paper are a description of the non-stable quintic threefolds in terms of singularities, partial description of the stable locus, and a complete description of boundary components and stratification of the GIT compactification using minimal orbits.  The first main result of the paper is given in section~\ref{sec:geointermaxsemi} which describes the non-stable quintic threefolds in terms of singularities.

\begin{theorem} \label{theorem:non-stable} A quintic threefold $X$ is non-stable if and only if one of the following properties holds: 
\begin{enumerate}
\item $X$ contains a double plane;
\item $X$ is a reducible variety, where a hyperplane is one of the components;
\item $X$ contains a triple line;
\item $X$ contains a quadruple point;
\item $X$ contains a triple point $p$ with the following properties:
\begin{enumerate}	
\item the tangent cone of $p$ is the union of a double plane and another hyperplane;
\item the line connecting a point in the double plane with the triple point has intersection multiplicity 5 with $X$;
\end{enumerate}
\item $X$ has a double line $L$ where every point $p \in L$ has the following properties:
\begin{enumerate}
\item the tangent cone of each point $p \in L$ is a double plane $P_p$;
\item each point $p \in L$ has the same double plane tangent cone i.e. $P_p=P$ for some double plane $P$; 
\item the line connecting the point on the tangent cone $P_p$ and a point $p \in L$ has intersection multiplicity 4 with $X$;
\end{enumerate}
\item $X$ contains a triple point $p$ and a plane $P$ with the following properties:
\begin{enumerate}
\item the tangent cone of $p$ contains a triple plane of $P$;
\item the singular locus of $X$, when restricted to $P$, is the intersection of two quartic curves $q_1$ and $q_2$;
\item the point $p$ is a quadruple point of $q_1$ and $q_2$.
\end{enumerate}
\end{enumerate}
\end{theorem}

A partial description of the stable locus is given in section~\ref{sec:stablelocus}.  The analysis of singularities of the non-stable quintic threefolds gives a partial description of the singularities that occur in the stable locus.  In particular, all smooth quintic threefolds and quintic threefold with at worst $A \hyph D \hyph E$ singularities will be GIT stable.  Following the approach of Laza \cite{laza}, the minimal orbits can be explicitly described using Luna's criterion\cite{luna,vinberg}.  In order to find the minimal orbits, the non-stable quintic threefolds degenerate into families of quintic threefolds given by equations~\ref{eq:MO-A}-\ref{eq:MO-D}, which are denoted the \textit{first level of minimal orbits}.  Luna's criterion determines which members of these families represent closed orbits, this is done in section~\ref{sec:firstlevelminimalorbits}.  Certain hypersurfaces in these families are unstable and therefore represent unstable quintic threefolds or they degenerate, even further, into a family of quintic threefolds given by equations~\ref{eq:MO2-I}-\ref{eq:MO2-X}.  The families represented by equations~\ref{eq:MO2-I}-\ref{eq:MO2-X} are called the \textit{second level of minimal orbits}.  The second level of minimal orbits represent how the boundary strata of the components in equations~\ref{eq:MO-A}-\ref{eq:MO-D} intersect.  Applying Luna's criterion to the second level of minimal orbits will determine which quintic threefolds are closed orbits, unstable orbits, and which hypersurfaces degenerate even further.  It will be shown in section~\ref{sec:secondlevelminimalorbits} that non-closed orbits in the second level of minimal orbits will eventually degenerate to the hypersurface $x_0x_1x_2x_3x_4$, which is the hypersurface with normal crossings singularities.  This completely determines the boundary structure and stratification of the GIT compactification of the moduli space of quintic threefolds.
\par
Section~\ref{sec:maxnon-stable} is devoted to the combinatorics and geometrical study of non-stable quintic threefolds.  Using the Hilbert-Mumford criterion, a combinatorial study of the monomials to be included in maximal non-stable families of quintic threefolds will be done.  Each maximal non-stable family has a destabilizing one-parameter subgroup (1-PS) $\lambda$ which gives rise to a ``bad flag'' that picks out the worst singularities in the family.  This is used to prove Theorem~\ref{theorem:non-stable}. The last part of section~\ref{sec:maxnon-stable} gives a partial description of the stable locus.  Section~\ref{sec:minimalorbit} introduces Luna's criterion and then applies it to study the closed orbits in equations~\ref{eq:MO-A}-\ref{eq:MO-D}.  Section~\ref{sec:boundary} studies the closed orbits and further degenerations of equations~\ref{eq:MO2-I}-\ref{eq:MO2-X}, thereby giving a complete description of the boundary stratification of the GIT compactification of quintic threefolds. 
\subsection*{Acknowledgements}  The author would like to Amassa Fauntleroy and Radu Laza for their advice and support in writing this paper.  The author would also like to thank Ryan Therkelsen for introducing the author to Stembridge's poset Maple package.
  
%%%%%%%%%%%%%%%%%%%%%%%%%%%%%%%%%%%%%%%%%%%%%%%%%%%%%%%%%%%%%%%%%%%%%%
\medskip

\section{Maximal Nonstable Families}
\label{sec:maxnon-stable}

\medskip

The Hilbert-Mumford criterion is an important tool when establishing stability and semistability in GIT.

\begin{remark}  Following the convention in \cite{laza}, a \textit{normalized 1-PS $\lambda$} is a map $\lambda: \CC^* \to T \subseteq SL(5,\CC)$, where $T$ is the the standard maximal torus $T$ of $SL(5,\CC)$, with the additional property that $\lambda(t) = diag(t^{a_0},t^{a_1},t^{a_2},t^{a_3},t^{a_4})$ satisfy $a_0 \geq a_1 \geq a_2 \geq a_3 \geq a_4$ and $a_0+a_1+a_2+a_3+a_4=0$.
\end{remark} 

For a quintic form $f \in \polyring$ and a normalized 1-PS $\lambda = \langle a_0,a_1,a_2,a_3,a_4 \rangle$, the numerical function is defined as follows,
\begin{equation}
\mu(f,\lambda):=max \{ a_0 i_0+a_1i_1+a_2i_2+a_3 i_3+a_4i_4 | f=\sum c_{i_0i_1i_2i_3i_4}x_0^{i_0}x_1^{i_1}x_2^{i_2}x_3^{i_3}x_4^{i_4}, c_{i_0i_1i_2i_3i_4} \neq 0 \}.
\end{equation}

The Hilbert-Mumford criterion states \textit{a quintic form f is stable (semistable) if and only if for every 1-PS $\lambda$ the numerical function $\mu(f,\lambda) > 0 ( \geq 0)$}. It can be restated so that a quintic form $f$ is non-stable (unstable) if there exists a 1-PS $\lambda$ where $\mu(f,\lambda) \leq 0$ $( < 0)$.  If quintic forms are analyzed up to coordinate transformation ($SL(5,\CC)$-action), then the $G$-equivariance of the numerical function\cite{standardmonomial}, 
\begin{equation} 
\label{eq:gequivariant}
\mu(x,\lambda) = \mu(gx, g \lambda g^{-1}),
\end{equation}
restricts to check checking the criterion for only normalized 1-PS $\lambda$ in the standard maximal torus $T$ of $SL(5,\CC)$.

\begin{remark}
In the remainder of the paper the a monomial $x_0^{i_0}x_1^{i_1}x_2^{i_2}x_3^{i_3}x_4^{i_4}$ will also be written as a vector denoted $[i_0,i_1,i_2,i_3,i_4]$.  Also, the following combinatorial procedure described below is based on similar combinatorial techniques described in Mukai \cite{mukai} and applied by Laza \cite{laza} in the case of cubic fourfolds. 
\end{remark}

The normalized 1-PS $\lambda$ induce a partial order on the set of quintic monomials $[i_0,i_1,i_2,i_3,i_4]$ given by
\[
[i_0,i_1,i_2,i_3,i_4] \geq [j_0,j_1,j_2,j_3,j_4] \Longleftrightarrow 
\begin{cases} 
  \mu([i_0,i_1,i_2,i_3,i_4],\lambda) \geq \mu([j_0,j_1,j_2,j_3,j_4],\lambda)  \\
  \mbox{for all normalized 1-PS } \lambda  
\end{cases}
\]
The following lemma is useful in creating an algorithm to determine the poset structure of quintic monomials.
\begin{lemma} [c.f. \cite{mukai} p.225] For two monomials $[i_0,i_1,i_2,i_3,i_4]$ and $[j_0,j_1,j_2,j_3,j_4]$
\[
[i_0,i_1,i_2,i_3,i_4] \geq [j_0,j_1,j_2,j_3,j_4] \Longleftrightarrow 
\begin{cases} 
  i_0 \geq j_0 \\
  i_0+i_1 \geq j_0+j_1  \\
  i_0+i_1+i_2 \geq j_0+j_1+j_2  \\
  i_0+i_1+i_2+i_3 \geq j_0+j_1+j_2+j_3  
  
\end{cases}
\]
\end{lemma}

This criterion is useful because one can directly check whether two monomials are related in the poset by checking the subsequent inequalities.  Using Maple\cite{maple}, the above criterion can be used to find all partial order relationships between monomials.   Stembridge's poset package for Maple \cite{stembridgemaple} is used to find the minimal covering relationships for these monomials and thereby creating the poset for quintic monomials.  The code for this entire procedure is given in appendix~\ref{appendix:posetcode}.  The figure for this poset structure is given in figure~\ref{fig:calabiyau} at the end of the paper.

\subsection{Combinatorics of Non-Stable Families}

The poset structure on quintic monomials greatly simplifies the Hilbert-Mumford criterion analysis on quintic polynomials $f$.  For a fixed normalized 1-PS $\lambda$ and a monomial $[i_0,i_1,i_2,i_3,i_4]$, $\mu([i_0,i_1,i_2,i_3,i_4], \lambda)$ does not change when any monomials below it in the poset are added to $[i_0,i_1,i_2,i_3,i_4]$. For a fixed normalized 1-PS $\lambda$, $\maxsemistabfam$ ($\maxunstabfam$) represent the set of monomials in the poset where $\mu \leq 0$ ($\mu <0$).  By the Hilbert-Mumford criterion, the monomials of every non-stable quintic polynomial $f$, up to coordinate transformation, belong to a family of the form $\maxsemistabfam$.  The \textit{maximal non-stable families}, denoted $SSk$, are the largest possible families of the form $\maxsemistabfam$.  The corresponding $\lambda$ of $\maxsemistabfam$ is called the family's \textit{destabilizing 1-PS}.  From the poset structure, there will be a finite number of maximal non-stable families $SSk$.  Any non-stable quintic form $f$, up to coordinate transformation; will be long to one of these families.  
\par

The procedure for determining the set of maximal non-stable families would be to start from the top monomial and work down the poset until one finds a monomial $[i_0,i_1,i_2,i_3,i_4]$ which has a normalized 1-PS $\lambda$ where $\mu([i_0,i_1,i_2,i_3,i_4], \lambda) \leq 0$.  By restricting to normalized $\lambda$ one can use linear programming to determine whether such a $\lambda$ exists for a particular monomial.  The linear programming script is given in appendix~\ref{appendix:linearprogram}.  Using this procedure it is determined that the top most monomials which have $\mu \leq 0$ are [3,0,0,2,0], [4,0,0,0,1], [2,0,3,0,0], and [1,4,0,0,0], the poset of monomials below these top monomials are denoted SS1, SS2, SS3, and SS4.  The figures for these posets are given in figures~\ref{fig:SS1},~\ref{fig:SS2},~\ref{fig:SS3}, and~\ref{fig:SS4} at the end of the paper.  The destabilizing 1-PS $\lambda$ is given in table 1.  Other maximal non-stable families are found by finding the top-most monomials in the families $SS1 \hyph SS4$ which have a common destabalizing 1-PS $\lambda$.  There are three other such families denoted $SS5 \hyph SS7$.  These three families have multiple maximal monomials.  The figures for these posets are given in figures~\ref{fig:SS5a},~\ref{fig:SS5b},~\ref{fig:SS6a},~\ref{fig:SS6b},~\ref{fig:SS7a} and~\ref{fig:SS7b} at the end of the paper.  
\par
The combinatorial procedure above determines the maximal families $\maxsemistabfam$ where $\mu \leq 0$.  By the Hilbert-Mumford criterion, every semistable or unstable hypersurface will transform, via a coordinate transformation, into one of maximal non-stable families $\maxsemistabfam$.  

\begin{remark}
$q_{a,b}(x_m,x_n \parallel x_o,x_p)$ represents polynomials which are a linear combination of degree $a$ monomials in $x_m$ and $x_n$ multiplied by a degree $b$ monomials in $x_o$ and $x_p$
\end{remark}

\begin{prop}
$X$ is non-stable if and only if it belongs, via a coordinate transformation, to a hypersurface in families found in table 1.
\end{prop}

\begin{table}[htdp]  
\centering
\renewcommand{\arraystretch}{1.00}
\scalebox{0.95}{
    \begin{tabular}{ | c | c | c | c | }
    \hline
    \emph{Family} & \emph{Destabilizing 1-PS Subgroup} & \emph{Maximal Monomial}    & \emph{Degeneration} \\ \hline
    
    SS1 & $\langle 2, 2, 2, -3, -3 \rangle$ & $[3,0,0,2,0]$  & $MO \hyph A$ \\ \hline
    
   \multicolumn{4}{| p{15cm} |}{$q_{3,2}(x_0,x_1,x_2 \parallel x_3,x_4) + q_{2.3}(x_0,x_1,x_2 \parallel x_3,x_4)+q_{1,4}(x_0,x_1,x_2 \parallel x_3,x_4)+q_5(x_3,x_4)$} \\ 
    \hline \hline
    
       SS2 & $\langle 1, 1, 1, 1, -4 \rangle$ & $[4,0,0,0,1]$  & $MO \hyph D$ \\ \hline 
    
    \multicolumn{4}{|p{15cm}|}{$x_4q_4(x_0,x_1,x_2,x_3,x_4)$} \\
    \hline \hline
  
        SS3 & $\langle 3, 3, -2, -2, -2 \rangle$ & $[2,0,3,0,0]$  & $MO \hyph A$ \\ \hline
    
    \multicolumn{4}{|p{15cm}|}{$q_{2,3}(x_0,x_1 \parallel x_2,x_3,x_4)+q_{1,4}(x_0,x_1  \parallel x_2,x_3,x_4)+q_5(x_2,x_3,x_4)$} \\ 
    \hline \hline

     SS4 & $\langle 4, -1, -1, -1, -1 \rangle$ & $[1,4,0,0,0]$  & $MO \hyph D$ \\ \hline
    
   \multicolumn{4}{|p{15cm}|}{$x_0q_4(x_1,x_2,x_3,x_4)+q_5(x_1,x_2,x_3,x_4)$} \\ 
    \hline \hline

    SS5 & $\langle 1, 0, 0, 0, -1 \rangle$ & $[0,5,0,0,0], [1,3,0,0,1], [2,1,0,0,2]$  & $MO \hyph B$ \\ \hline 
    
    \multicolumn{4}{|p{15cm}|}{$ x_0^2 ( x_4^2q_1(x_1,x_2,x_3,x_4) ) + x_0x_4q_3(x_1,x_2,x_3,x_4) + q_5(x_1,x_2,x_3,x_4)$} \\ 
    \hline \hline

    SS6 & $\langle 4, 4, -1, -1, -6 \rangle$ & $[1,0,4,0,0], [3,0,0,0,2], [2,0,2,0,1] $  &  $MO \hyph C$ \\ \hline 
    
    \multicolumn{4}{|p{15cm}|}{$  x_4^2q_{3} (x_0,x_1)   +   x_4 q_{2,2}(x_0,x_1 \parallel x_2,x_3,x_4)  + q_{1,4}(x_0,x_1 \parallel x_2,x_3,x_4) + q_5(x_2,x_3,x_4)$} \\ 
    \hline \hline

    SS7 & $\langle 6, 1, 1, -4, -4 \rangle$ & $[0,4,0,1,0] [1,2,0,2,0] [2,0,0,3,0]   $ & $MO \hyph C$ \\ \hline 
    
    \multicolumn{4}{|p{15cm}|}{$ x^{2}_0q_3(x_3,x_4) +x_0  ( q_{2,2}(x_1,x_2 \parallel x_3,x_4)+q_{1,3}(x_1,x_2 \parallel x_3,x_4)+q_4(x_3,x_4) )  +  q_{4,1}(x_1,x_2 \parallel x_3,x_4)+q_{3,2}(x_1,x_2 \parallel x_3,x_4) +q_{2,3}(x_1,x_2 \parallel x_3,x_4)+q_{1,4}(x_1,x_2 \parallel x_3,x_4)+q_5(x_3,x_4). $ } \\ 
    \hline \hline

    \end{tabular}
\label{maxfam}

}
	\vspace{0.1cm}
\caption{Nonstable Families SS1 - SS7}
\end{table}

For any non-stable $f$, the $SL(5,\CC)$-orbit will not necessarily be closed.  Using the destabilizing 1-PS $\lambda$, the closure $\overline{\lambda f}$ = $f_0$ is a quintic form invariant with respect to $\lambda$.  The forms $f$ and $f_0$ will map to the same point in the GIT quotient.  If the orbit of $f_0$ is closed then it is a minimal orbit. 
\begin{eqnarray}
\label{eq:MO-A}  
 MO \hyph A:& q_{2,3}(x_0,x_1 \parallel x_2 , x_3, x_4), \\
\label{eq:MO-B}  
 MO \hyph B:& q_5(x_1,x_2,x_3)+x_0x_4q_3(x_1,x_2,x_3)+x_0^2x_4^2q_1(x_1,x_2,x_3), \\
\label{eq:MO-C}  
 MO \hyph C:& q_{1,4}(x_0,x_1 \parallel x_2,x_3) + x_4q_{2,2}(x_0,x_1 \parallel x_2,x_3) + x_4^2q_3(x_0,x_1), \\
\label{eq:MO-D}  
 MO \hyph D:& x_0q_4(x_1,x_2,x_3,x_4).
\end{eqnarray}

In Section~\ref{sec:minimalorbit}, it will be shown that a generic member of one of the families $MO \hyph A - MO \hyph D$ will represent a minimal orbit.  Certain hypersurfaces in $MO \hyph A - MO \hyph D$ will degenerate further into a member of one of the families $MO2 \hyph I - MO2 \hyph X$.

\begin{eqnarray}
\label{eq:MO2-I}  
 MO2 \hyph I:&  x_0^2  x_2x_4^2    + x_0x_1  x_2x_3x_4         + x_1^2  x_2x_3^2   , \\
\label{eq:MO2-II}  
 MO2 \hyph II:&  x_0^2  x_3x_4^2   + x_0x_1 (  x_3^3 + x_2x_3x_4 ) + x_1^2  x_2^2x_3  , \\
\label{eq:MO2-III}  
 MO2 \hyph III:&  x_0^2 ( x_2x_4^2 + x_3^2x_4  ) + x_0x_1 (  x_3^3 + x_2x_3x_4 ) + x_1^2 ( x_2x_3^2 + x_2^2x_4 ), \\
\label{eq:MO2-IV}  
 MO2 \hyph IV:& x_0x_1x_4  q_2(x_2,x_3) , \\
\label{eq:MO2-V}  
 MO2 \hyph V:& x_0x_1  x_2x_3x_4 , \\
\label{eq:MO2-VI}  
 MO2 \hyph VI:&     x_1x_2^3x_3 + x_1^2x_2x_3^2    + x_0 x_1  x_2x_3x_4  , \\
\label{eq:MO2-VII}  
 MO2 \hyph VII:& x_4x_0^2x_3^2 + x_4x_0x_1x_2x_3 + x_4x_1^2x_2^2  , \\
\label{eq:MO2-VIII}  
 MO2 \hyph VIII:&  x_0x_1q_3(x_2,x_3,x_4), \\
\label{eq:MO2-IX}  
 MO2 \hyph IX:& x_0q_{2,2}(x_1,x_2 \parallel x_3,x_4) , \\
\label{eq:MO2-X}  
 MO2 \hyph X:& x_0  (q_4(x_2,x_3) + x_1x_4q_2(x_2,x_3) + x_1^2x_4^2 ). 
 \end{eqnarray}

\subsection{Bad Flags}

The maximal non-stable families will be characterized in terms of singularities found on a generic member of one of these families.  A destabilizing 1-PS $\lambda$ has an associated ``bad flag'' of the vector spaces $H^0(\PP^4,\Oh_{\PP^4}(1)) \cong \CC^5$.  A general principal, given by Mumford \cite{mumford}, states that these ``bad flags'' pick out the singularities which cause the family to become semistable or unstable.  

\par

Using the approach given by Laza \cite{laza} it can be shown that a 1-PS $\lambda: \CC^* \to T$ gives a weight decomposition of $H^0(\PP^4,\Oh_{\PP^4}(1)) = \oplus_{i=0}^{5} W_i$ based on the eigenvalues of $\lambda$ acting on $H^0(\PP^4,\Oh_{\PP^4}(1))$. 

\begin{definition} For a 1-PS $\lambda = \langle a,b,c,d,e \rangle$ let $m_i$ be a subset of $\left\{ a,b,c,d,e \right\}$ which have the same weights and let $n_i$ be the weight.  
\begin{equation} W_{m_i} := \bigoplus_{\text{$i$ where $W_i$ has eigenvalue $n_i$}} W_i \end{equation}

\end{definition}

The standard flag is

\begin{equation} \label{eq:standardflagvariety} \begin{split}  \emptyset \subseteq   F_1 = ( x_1=x_2&=x_3=x_4=0 ) \subseteq  F_2 = ( x_2=x_3=x_4=0 ) \subseteq \\ & F_3 = ( x_3=x_4=0 ) \subseteq F_4 = (x_4=0) \subseteq \PP^4.  \end{split} \end{equation}

\begin{definition} Given a 1-PS $\lambda = \langle a,b,c,d,e \rangle$ let $m_1$, $m_2 \ldots$, $m_s$ represent the collection of common weights of $\lambda$.  Let $m_i$ be ordered by increasing value of weights (i.e. $m_1$ has lowest weight).  The \textit{associated flag for $\lambda$} is

\begin{equation} \label{eq:1psflagvariety} F_{\lambda}:\ \  \emptyset \subseteq  F_{m_{s}}:= \bigoplus_{i=1}^{s} W_{m_i} \subset F_{m_{s-1}}:= \bigoplus_{i=1}^{s-1} W_{m_i} \subset \ldots  \subseteq F_{m_{1}}:= W_{m_1} \subseteq \PP^4. \end{equation}

This is a subflag of the standard flag (\ref{eq:1psflagvariety}).

\end{definition}

For the maximal destabilizing families $SS1 \hyph SS7$ the associated ``bad flags'' are

\begin{table}[htdp]
\begin{center}
\renewcommand{\arraystretch}{1.25}
\scalebox{.93}{
    \begin{tabular}{ | c | c | c | }
    \hline
    \emph{Family} & \emph{Destabilizing 1-PS Subgroup} & \emph{Destabilizing Flag $F_\lambda$} \\ 
    \hline
    SS1 & $\langle 2, 2, 2, -3, -3 \rangle$ & $ \emptyset \subseteq (x_3=x_4=0)  \subseteq \PP^4$ \\ \hline
    
       SS2 & $\langle 1, 1, 1, 1, -4 \rangle$ & $ \emptyset \subseteq (x_4=0)  \subseteq \PP^4      $ \\ \hline 
    
      SS3 & $\langle 3, 3, -2, -2, -2 \rangle$ & $ \emptyset \subseteq (x_2=x_3=x_4=0)  \subseteq \PP^4    $ \\ \hline
    
     SS4 & $\langle 4, -1, -1, -1, -1 \rangle$ & $ \emptyset \subseteq (x_1=x_2=x_3=x_4=0)  \subseteq \PP^4    $ \\ \hline
    
   SS5 & $\langle 1, 0, 0, 0, -1 \rangle$ & $ \emptyset \subseteq (x_1=x_2=x_3=x_4=0) \subseteq (x_4=0)  \subseteq \PP^4 $ \\ \hline 
    
    SS6 & $\langle 4, 4, -1, -1, -6 \rangle$ & $ \emptyset \subseteq (x_2=x_3=x_4=0) \subseteq (x_4=0)  \subseteq \PP^4      $ \\ \hline 
    
 SS7 & $\langle 6, 1, 1, -4, -4 \rangle$ & $ \emptyset \subseteq (x_1=x_2=x_3=x_4=0) \subseteq (x_3=x_4=0)  \subseteq \PP^4 .   $ \\ \hline 

    \hline

    \end{tabular}
    }
    	\vspace{0.1cm}
\caption{Destabilizing Flags of SS1-SS7}
\end{center}
\end{table}

\subsection{Geometric Interpretation of Maximal Semistable Families} \label{sec:geointermaxsemi}

In order to determine the singularities which occur on threefolds in families $SS1 \hyph SS7$, we intersect the general form of the equation with its associated destabilizing flag.  This will give some description of the types of singularities, which occur in each family.  A precise description of each such family is given in the propositions below.  Some of the singularity analysis is based on describing the tangent cone and intersection multiplicities of the tangent cone at singular points, a detailed introduction of these topics is given in Beltrametti et al. (\cite{beltrametti} ch.5)

\begin{prop}
A hypersurface $X$ is of type SS1 if and only if $X$ contains a double plane.
\end{prop}

\begin{proof}
\indent
Let $X$ be of type $S1$ then it is equivalent, via a coordinate transformation, to the hypersurface 
\begin{equation}
\begin{split}
q_{3,2}(x_0,x_1,x_2 \parallel x_3,x_4) + q_{2.3}(x_0,x_1,x_2 \parallel x_3,x_4) \\ + q_{1,4}(x_0,x_1,x_2 \parallel x_3,x_4)+q_5(x_3,x_4).  
\end{split}
\label{SS1}
\end{equation}
This hypersurface contains the ideal $\langle x_3,x_4 \rangle^2$ which is a double plane in $\PP^4$.

\indent

Let $X$ be a hypersurface which contains a double plane.  By a coordinate transformation we can assume the double plane is $\langle x_3,x_4 \rangle^2$.  The most general equation which contains the ideal $\langle x_3,x_4 \rangle^2$ is (\ref{SS1}).

\end{proof}

\begin{prop}
A hypersurface $X$ is of type SS2 if and only if $X$ is a reducible variety, where a hyperplane is one of the components.  In particular, the singularity is the intersection of the hyperplane with the other component, which is generically a degree 4 surface.
\end{prop}
\begin{proof}
\indent
Let $X$ be of type $S2$ then it is equivalent, via a coordinate transformation, to the hypersurface 

\begin{equation}
\begin{split}
x_4q_4(x_0,x_1,x_2,x_3,x_4)
\end{split}
\label{SS2}
\end{equation}
This hypersurface has the hyperplane $\langle x_4 \rangle$ as a component.

\par

Let $X$ be a reducible hypersurface where a hyperplane is a component.  The polynomial $f \in \CC[x_0,x_1,x_2,x_3,x_4]$ defining $X$ can be factored into $f = gh$, where $h$ is a degree 1 polynomial.  By a coordinate transformation we can map the hyperplane defining $h$ to $x_4$.  Without loss of generality $f=x_4h$. Since since $f$ is of degree 5 then by neccesity $h$ is of degree 4 therefore $f$ is of the form (\ref{SS2}).

\end{proof}

\begin{prop}
A hypersurface $X$ is of type SS3 if and only if $X$ contains a triple line.
\end{prop}

\begin{proof}
\indent
Let $X$ be of type $SS3$ then it is equivalent, via a coordinate transformation, to the hypersurface 
\begin{equation}
\begin{split}
q_{2,3}(x_0,x_1 \parallel x_2,x_3,x_4)+q_{1,4}(x_0,x_1  \parallel x_2,x_3,x_4)+q_5(x_2,x_3,x_4)
\end{split}
\label{SS3}
\end{equation}
This hypersurface contains the ideal $\langle x_2,x_3,x_4 \rangle^3$ which is a triple line in $\PP^4$.

\indent

Let $X$ be a hypersurface which contains a triple line.  By a coordinate transformation, we can assume the triple line is $\langle x_2,x_3,x_4 \rangle^3$.  The most general equation which contains the ideal $\langle x_2,x_3,x_4 \rangle^3$ is (\ref{SS3}).

\end{proof}

\begin{prop}
A hypersurface $X$ is of type $SS4$ if and only if $X$ contains a quadruple point.
\end{prop}

\begin{proof}
\indent
Let $X$ be of type $SS4$ then it is equivalent, via a coordinate transformation, to the hypersurface 
\begin{equation}
\begin{split}
x_0q_4(x_1,x_2,x_3,x_4)+q_5(x_1,x_2,x_3,x_4)
\end{split}
\label{SS4}
\end{equation}
This hypersurface contains the ideal $\langle x_1,x_2,x_3,x_4 \rangle^4$ which is a quadruple point in $\PP^4$.

\par

Let $X$ be a hypersurface which contains a quadruple point.  By a coordinate transformation we can assume the quadruple point is $\langle x_1,x_2,x_3,x_4 \rangle^4$.  The most general equation which contains the ideal $\langle x_1,x_2,x_3,x_4 \rangle^4$ is (\ref{SS4}).

\end{proof}

\begin{prop}
A hypersurface $X$ is of type $SS5$ if and only if $X$ has a triple point $p$ with the following properties:
\begin{itemize}
\item [i)] the tangent cone of $p$ is the union of a double plane and another hyperplane;
\item [ii)] the line connecting a point in the double plane with the triple point has intersection multiplicty 5 with the hypersurface.
\end{itemize}
\end{prop}

\begin{proof}
\indent
Let $X$ be of type $SS5$ then it is equivalent, via a coordinate transformation, to the hypersurface 
\begin{equation}
\begin{split}
x_0^2 \bigg( x_4^2q_1(x_1,x_2,x_3,x_4) \bigg) + x_0x_4q_3(x_1,x_2,x_3,x_4) + q_5(x_1,x_2,x_3,x_4)
\end{split}
\label{SS5}
\end{equation}
This hypersurface contains the triple point $\langle x_1,x_2,x_3,x_4 \rangle^3$.  The tangent cone is the hypersurface defined by

\begin{equation}
\begin{split}
x_4^2q_1(x_1,x_2,x_3,x_4) 
\end{split}
\end{equation}

which is the union of a double hyperplane $\langle x_4 \rangle^2$ and another general hyperplane $q_1(x_1,x_2,x_3,x_4)$.  The points whose lines passing through the triple point which have intersection multiplicity 5 with the hypersurface, is the locus of $\langle x_4^2q_1(x_1,x_2,x_3,x_4) \rangle$ and $\langle x_4q_3(x_1,x_2,x_3,x_4) \rangle$.  Since $x_4$ is a component of both terms then a line emanating from the hyperplane $\langle x_4 \rangle$ to the triple point will have multiplicity 5. 

\par

Let $X$ be a hypersurface which contains a triple point.  By a coordinate transformation we can assume the triple point is $\langle x_1,x_2,x_3,x_4 \rangle^3$.  The most general equation which contains the ideal $\langle x_1,x_2,x_3,x_4 \rangle^3$ is

\begin{equation}
\begin{split}
x_0^2\bigg( q_3(x_1,x_2,x_3,x_4) \bigg) + x_0 \bigg( q_4(x_1,x_2,x_3,x_4) \bigg) + q_5(x_1,x_2,x_3,x_4).
\end{split}
\end{equation}

If the tangent cone is the union of a double plane and another hyperplane then

\begin{equation}
\begin{split}
q_3(x_1,x_2,x_3,x_4)=f^2g
\end{split}
\end{equation}

where $f$ and $g$ are linear forms.  By a coordinate transformation which keeps the triple point fixed we can map the hyperplane $f$ to $x_4$.  So without loss of generality

\begin{equation}
\begin{split}
q_3(x_1,x_2,x_3,x_4)=x_4^2g.
\end{split}
\end{equation}

If a general line from the hyperplane $\langle x_4 \rangle$ to the triple point has multiplicity 5 then 

\begin{equation}
\begin{split}
x_4=q_4=0.  
\end{split}
\end{equation}

This occurs only if $q_4$ has $x_4$ as a component so

\begin{equation}
\begin{split}
q_4=x_4q_3 
\end{split}
\end{equation}

which is precisely of the form (\ref{SS5}).

\end{proof}

\begin{prop}
A hypersurface $X$ is of type $SS6$ if and only if $X$ has a double line $L$ where every point $p \in L$ has the following properties:

\begin{itemize}
\item[i)] the tangent cone of each point $p \in L$ is a double plane $P_p$;
\item [ii)]  each point $p \in L$ has the same double plane tangent cone i.e. $P_p=P$ for some double plane $P$; 
\item [iii)] the line connecting the point on the tangent cone $P_p$ and a point $p \in L$ has intersection multiplicity 4 with the hypersurface.
\end{itemize}
\end{prop}

\begin{proof}
\indent
Let $X$ be of type $SS6$ then it is equivalent, via a coordinate transformation, to the hypersurface 
\begin{equation}
\begin{split}
\bigg( q_{3} (x_0,x_1)x_4^2 \bigg)  +     \bigg( x_4 q_{2,2}(x_0,x_1 \parallel x_2,x_3,x_4) \bigg) \\ + q_{1,4}(x_0,x_1 \parallel x_2,x_3,x_4) + q_5(x_2,x_3,x_4)
\end{split}
\label{SS6}
\end{equation}
This hypersurface contains the double line  $\langle x_2,x_3,x_4 \rangle^2$.  For any point $[\lambda: \nu: 0:0:0]$ of the double line the tangent cone is the same double plane given by $\langle x_4 \rangle^2$.  The points which have intersection multiplicity 4 with the double line are the locus of $\langle x_4 \rangle^2 $ and $\langle x_4 q_{2,2}(\lambda, \nu \parallel x_2,x_3,x_4) \rangle$.  Since $x_4$ is a component of both terms then the line emanating from the hyperplane $\langle x_4 \rangle$ to any point of the double line will have multiplicity 4. 

\par

Let $X$ be a hypersurface which contains a double line.  By a coordinate transformation we can assume the double line is $\langle x_2,x_3,x_4 \rangle^2$.  The most general equation which contains the ideal $\langle x_2,x_3,x_4 \rangle^2$ is

\begin{equation}
\begin{split}
q_{3,2}(x_0,x_1 \parallel x_2,x_3,x_4) + q_{2,3}(x_0,x_1 \parallel x_2,x_3,x_4) + q_{1,4}(x_0,x_1 \parallel x_2,x_3,x_4) + q_5(x_2,x_3,x_4).  
\end{split}
\end{equation}

If the tangent cone at every point on the double line is the same double plane then

\begin{equation}
\begin{split}
q_{3,2}(x_0,x_1 \parallel x_2,x_3,x_4) = q_3(x_0,x_1)f(x_2,x_3,x_4)^2 
\end{split}
\end{equation}

where $f$ is a linear form.  By a coordinate transformation, which keeps the double line fixed, the hyperplane $f$ is mapped to $x_4^2$.  So without loss of generality,

\begin{equation}
\begin{split}
q_{3,2}(x_0,x_1 \parallel x_2,x_3,x_4)=q_{3} (x_0,x_1)x_4^2.  
\end{split}
\end{equation}

If the line going from the hyperplane $\langle x_4 \rangle$ to any point of the double line has multiplicity 4 then

\begin{equation}
\begin{split}
x_4=q_{2,3}(\lambda,\nu \parallel x_2,x_3,x_4)=0.  
\end{split}
\end{equation}

This occurs only if $q_{2,3}(x_0,x_1 \parallel x_2,x_3,x_4)$ has $x_4$ as a component so

\begin{equation}
\begin{split}
q_{2,3}=x_4q_{2,2}(x_0,x_1 \parallel x_2,x_3,x_4) 
\end{split}
\end{equation}

which is precisely of the form (\ref{SS6}).

\end{proof}

\begin{prop}
A hypersurface $X$ is of type $SS7$ if and only if $X$ contains a triple point $p$ and a plane $P$, where $p \in P$ has the following properties:
\begin{itemize}
\item [i)] the tangent cone of $p$ contains a triple plane of $P$;
\item [ii)] the singular locus of $X$, when restricted to $P$, is the intersection of two quartic curves $q_1$ and $q_2$;
\item [iii)] the point $p$ is a quadruple point of $q_1$ and $q_2$.
\end{itemize}
\end{prop}

\begin{proof}
\indent
Let $X$ be of type $SS7$ then it is equivalent, via a coordinate transformation, to the hypersurface 
\begin{equation}
\begin{split}
&x^{2}_0q_3(x_3,x_4) +x_0  \bigg( q_{2,2}(x_1,x_2 \parallel x_3,x_4)+q_{1,3}(x_1,x_2 \parallel x_3,x_4)+q_4(x_3,x_4) \bigg) \\ +  &\bigg( q_{4,1}(x_1,x_2 \parallel x_3,x_4)+q_{3,2}(x_1,x_2 \parallel x_3,x_4) +q_{2,3}(x_1,x_2 \parallel x_3,x_4)\\ + &q_{1,4}(x_1,x_2 \parallel x_3,x_4)  +q_5(x_3,x_4) \bigg)
\end{split}
\label{SS7}
\end{equation}
This hypersurface contains the triple point $p$ given by the ideal $\langle x_1,x_2,x_3,x_4 \rangle^3$ and a plane $P$ given by $\langle x_3,x_4 \rangle$.  The tangent cone is the hypersurface defined by $q_3(x_3,x_4)$ which which contains the triple plane $\langle x_3,x_4 \rangle^3$ of $P$.  When the differential of $X$ is restricted to the plane $\langle x_3,x_4 \rangle$ the only non-trivial contribution comes from the term 

\begin{equation}
\begin{split}
q_{4,1}(x_1,x_2 \parallel x_3,x_4) = q_4(x_1,x_2)x_3+ \tilde{q_4}(x_1,x_2)x_4.
\end{split}
\end{equation}

The differential, when restricted to the plane, is zero when

\begin{equation}
\begin{split}
q_4(x_1,x_2)=\tilde{q_4}(x_1,x_2)=0.
\end{split}
\end{equation}

Therefore, the plane contains two quartic curves $q_4(x_1,x_2)$ and $\tilde{q_4}(x_1,x_2)$ which contain $p$ as the quadruple point. 

\par

Let $X$ be a hypersurface which contains a triple point $p$ and a plane $P$, where $p \in P$.  By a coordinate transformation we can assume the triple point is $\langle x_1,x_2,x_3,x_4 \rangle^3$ and the plane is $\langle x_3,x_4 \rangle$.  The most general equation which contains the ideal $\langle x_1,x_2,x_3,x_4 \rangle^3$ and $\langle x_3,x_4 \rangle$ is 

\begin{equation}
\begin{split}
 x^{2}_0 & \bigg( q_{2,1}(x_1,x_2 \parallel x_3,x_4) +q_{1,2}(x_1,x_2 \parallel x_3,x_4)+ q_3(x_3,x_4) \bigg) \\   +  & x_0  \bigg( q_{3,1}(x_1,x_2 \parallel x_3,x_4)  + q_{2,2}(x_1,x_2 \parallel x_3,x_4)+q_{1,3}(x_1,x_2 \parallel x_3,x_4)+q_4(x_3,x_4) \bigg) \\  + & \bigg( q_{4,1}(x_1,x_2 \parallel x_3,x_4)+q_{3,2}(x_1,x_2 \parallel x_3,x_4) +q_{2,3}(x_1,x_2 \parallel x_3,x_4)\\ +&q_{1,4}(x_1,x_2 \parallel x_3,x_4)+q_5(x_3,x_4) \bigg)
\end{split}
\label{SS7converse}
\end{equation}

If the tangent cone contains the triple plane of $P$ then it contains the ideal $\langle x_3,x_4 \rangle^3$.  Then the coefficients of the $x_0^2$ term of (\ref{SS7converse}) contains only the $q_3(x_3,x_4)$ term.  The differential of (\ref{SS7converse}), when restricted to the plane $\langle x_3,x_4 \rangle$, contains the equations of the form $x_0q_3(x_1,x_2)+q_4(x_1,x_2)$ and $x_0\tilde{q_3}(x_1,x_2)+\tilde{q_4}(x_1,x_2)$.  If the singular locus of $X$ in the plane is the intersection of two quartic curves then

\begin{equation}
\begin{split}
x_0q_3(x_1,x_2)+q_4(x_1,x_2) = x_0 \tilde{q_3}(x_1,x_2)+ \tilde{q_4}(x_1,x_2)=0.
\end{split}
\end{equation}

So $x_0q_3(x_1,x_2)+q_4(x_1,x_2)$ and $x_0\tilde{q_3}(x_1,x_2)+\tilde{q_4}(x_1,x_2)$ are the quartic curves.  If $p$ is a quadruple point of both quartic curves then $q_3$ and $\tilde{q_3}$ are 0, so $X$ is of the form ($\ref{SS7}$).

\end{proof}

This completes the proof of Theorem~\ref{theorem:non-stable}.

\subsection{Stable Locus} \label{sec:stablelocus}

The classification of singularities of non-stable quintic threefolds can be used to give a partial description of the singularities which occur in the stable locus.  The stable locus represents all of the closed orbits in the moduli space.   Ideally, the stable locus would only include smooth hypersurfaces and the boundary would include hypersurfaces with singularities.  Even in the case of cubic threefolds and cubic fourfolds, this is not the case as shown in\cite{laza,allcock}.  As the degree and dimension of hypersurfaces increases more singularity types will be included in the stable locus.  In \cite{mumford} there is a general proposition which states that a smooth hypersurface will always be stable.  

\begin{prop}[\cite{mumford} Prop. 4.2] A smooth hypersurface $F$ in $\PP^n$ with degree $\geq 2$ is a stable hypersurface.
\end{prop}

A complete classification of all possible singularities in the stable locus of the moduli space of quintic threefolds has not yet been found.  Using the results of the previous section a partial list of singularities can be determined.

\begin{prop} If $X$ is a quintic threefold with at worst a double point then it is stable.
\end{prop}
\begin{proof} 
Suppose $X$ is not stable, then it is non-stable.  Therefore, it belongs to one of the families SS1 - SS7, but $X$ does not satisfy the singularity criteria for any of these families.  Hence, it is stable.
\end{proof}

\begin{prop} If $X$ is a quintic threefold with at worst a triple point whose tangent cone is an irreducible cubic surface and $X$ does not contain a plane then it is stable.
\end{prop}
\begin{proof} 
Suppose $X$ is not stable, then it is non-stable.  Therefore, it belongs to one of the families SS1 - SS7.  The only families, which have at worst a triple point are families SS5 and SS7.  Since the tangent cone of $X$ is irreducible then it is not in SS5.  Since $X$ does not contain a plane it is not in SS7, therefore it belong to neither family.  Hence, it is stable.
\end{proof}

\begin{prop} If $X$ is a quintic threefold with at worst a double line whose tangent cone at each point on the line is irreducible then it is stable.
\end{prop}
\begin{proof} 
Suppose $X$ is not stable, then it is non-stable.  So it belongs to one of the families SS1 - SS7.  SS6 is the only family, which has at worst a double line as a singularity.  Since the tangent cone of $X$ at each point is irreducible then it is not in SS6.  Hence, it is stable.
\end{proof}

These four classes of hypersurfaces give the most generic classes of hypersurfaces which are stable.  There are also quintic threefolds which have singularities with degenerate tangent cones that do not fit into one of the classes SS1-SS7, but a complete classification is still unknown.

%%%%%%%%%%%%%%%%%%%%%%%%%%%%%%%%%%%%%%%%%%%%%%%%%%%%%%%%%%%%%%%%%%%%%%
\medskip
\section{Minimal Orbits}

\label{sec:minimalorbit}

\medskip

\subsection{Luna's Criterion} \label{sec:lunascriterion}

The ability to degenerate the large families $SS1 - SS7$ into much smaller invariant families $MO\hyph A - MO \hyph D$ makes the problem of finding minimal orbits much more tractable.  Generically, a hypersurface in one of the families $MO\hyph A - MO \hyph D$ will be closed and thus minimal.  To explicitly determine which elements in $MO\hyph A - MO\hyph D$ are closed and which elements further degenerate one can use Luna's criterion.  This approach was used by Laza \cite{laza} in the case of cubic fourfolds.

\par
Luna's criterion is used when there is an affine $G$-variety $Y$ and a point $y \in Y$ which has a non-finite stabilizer $H \subseteq G$.  If $Y^H$ is the set of points in $Y$ which are $H$-invariant and $N_G(H)$ is the normalizer of $H$ in $G$ then there is a natural action of $N_G(H)$ on $Y^H$.  Luna's criterion reduces the problem of determining whether $Gy$ is closed in $Y$ to whether $N_G(H)$ is closed in $Y^H$.  

\begin{prop}[Luna's Criterion] \cite{vinberg,luna}

Let $Y$ be an affine variety with a $G$-action and $y \in Y$ a point stabilized by a subgroup $H \subseteq G$.  Then the orbit $Gy$ is closed in $Y$ if and only if the orbit $N_G(H)y$ is closed in $Y^H$.
\end{prop}

\begin{remark}[\cite{vinberg}]
In the case where $H$ is reductive and connected, $N_G(H) = H \cdot Z_G(H)$ where $Z_G(H)$ is the centralizer of $H$ in $G$.  Since $N_G(H)$ acts on $Y^H$ we can quotient out by $H$.  Thus, we can study the action of $Z_G(H)$, instead of $N_G(H)$, on $Y^H$.
\end{remark}

The case of quintic threefolds consists of an $SL(5,\CC)$-action on the projective variety $\PP(V)$.  $V$ is the linearization of the $SL(5,\CC)$-action on $\PP(V)$.  The closed orbits of points in the linearization $V$ correspond to closed orbits of points in $\PP(V)$.  This correspondence between a projective variety and its linearization allows us to apply Luna's criterion to $V$.  Given a point from one of the families $MO\hyph A- MO \hyph D$, the stabilizer subgroup is the invariant 1-PS  i.e. $H=\lambda$.  The following lemma reduces the problem of finding minimal orbits to finding stable points in the $Z_G(H)$-action on $V^H$.

\begin{lemma}
Let $v\in V$ be a point with stabilizer $H$ i.e. $v \in V^H$. If $v\in V$ is stable with respect to the $Z_G(H)$-action on the $H$-invariant space $V^H$ then the orbit $Gv$ is closed.
\end{lemma} 

\begin{proof}
Let $v\in V^H$ be stable with respect to the $Z_G(H)$-action on $V^H$.  By the definition of stable point, the orbit $Z_G(H)v$ is closed.  By Luna's criterion, $Gv$ is closed.
\end{proof}

\subsection{First Level of Minimal Orbits} \label{sec:firstlevelminimalorbits}

The centralizer groups for the families $MO\hyph A - MO\hyph D$ are given in table 3.  If $v \in V$ from $MO\hyph A- MO \hyph D$ is $Z_G(H)$-stable in $V^H$, then it is a minimal orbit.  If that point $v$ is non-stable, with respect to the $Z_G(H)$-action, then it is either unstable or there is a destablizing 1-PS $\lambda$.  This destabilizing 1-PS further degenerates $v$ into a smaller family, with a different stabilizer $H'$.  The same process is then repeated for the smaller families.  The Hilbert-Mumford criterion, with respect to the $Z_G(H)$-action on $V^H$, can be applied to determine precisely the stable, semistable, and unstable points.

\begin{table}[htdp]
\centering
\renewcommand{\arraystretch}{1.0}
\scalebox{0.95}{
    \begin{tabular}{ | c | c | c | }
    \hline
    \emph{Family} & \emph{Invariant 1-PS Subgroup (H)} & \emph{Centralizer of H ($Z_G(H))$} \\  \hline
    MO-A & $\langle 3, 3, -2, -2, -2 \rangle$ & $SL(2,\CC) \times SL(3,\CC)$ \\ \hline
    
   \multicolumn{3}{| p{14cm} |}{$q_{2,3}(x_0,x_1 \parallel x_2 , x_3, x_4)$} \\ 
    \hline \hline

    MO-B & $\langle 1, 0, 0, 0, -1 \rangle$ & $\CC^* \times SL(3,\CC) \times \CC^*$ \\ \hline
    
   \multicolumn{3}{| p{14cm} |}{$q_5(x_1,x_2,x_3)+x_0x_4q_3(x_1,x_2,x_3)+x_0^2x_4^2q_1(x_1,x_2,x_3)$} \\ 
    \hline \hline

    MO-C & $\langle 4, 4, -1, -1, -6 \rangle$ & $SL(2,\CC) \times SL(2,\CC) \times \CC^*$ \\ \hline
    
   \multicolumn{3}{| p{14cm} |}{$q_{1,4}(x_0,x_1 \parallel x_2,x_3) + x_4q_{2,2}(x_0,x_1 \parallel x_2,x_3) + x_4^2q_3(x_0,x_1) 	$} \\ 
    \hline \hline

    MO-D & $\langle 4, -1, -1, -1, -1 \rangle$ & $\CC^* \times SL(4,\CC) $ \\ \hline
    
   \multicolumn{3}{| p{14cm} |}{$x_0q_4(x_1,x_2,x_3,x_4) 	$} \\ 
    \hline \hline

    \end{tabular}
    }
	\vspace{0.1cm}
\caption{First Level of Minimal Orbits MO-A - MO-D}
\label{table:firstminorbit}
\end{table}

\subsubsection{Minimal Orbit A} \label{sec:minorbita}

In the case of family $MO \hyph A$, the centralizer $Z_G(H)=\CC^2 \times SL(3,\CC) \subseteq SL(5,\CC)$ acts on polynomials in $MO\hyph A$ which are of the form $q_{2,3}(x_0,x_1 \parallel x_2,x_3,x_4)$.  The minimal orbit can be written as 

\begin{equation} \label{eq:moadecomp} x_0^2s_3(x_2,x_3,x_4) + x_0x_1t_3(x_2,x_3,x_4) + x_1^2u_3(x_2,x_3,x_4),
\end{equation}
where $s_3(x_2,x_3,x_4)$, $t_3(x_2,x_3,x_4)$, $u_3(x_2,x_3,x_4)$ are degree $3$ polynomials in the variables $x_2$, $x_3$, and $x_4$.  The polynomials which represent closed orbits in this family are stable with respect to the $SL(2,\CC) \times SL(3,\CC)$ action on (\ref{eq:moadecomp}).  The polynomials that further degenerate are semistable with respect to the $SL(2,\CC) \times SL(3,\CC)$ action.  The unstable points are unstable with respect to the $SL(2,\CC) \times SL(3,\CC)$ action.  
The set of semistable and unstable points can be found by modifying the techniques in section~\ref{sec:maxnon-stable} which involved classifying $G$-orbits by using (\ref{eq:gequivariant}) to find which polynomials which satisfy the Hilbert-Mumford criterion for the set of normalized 1-PS $\lambda$.  A normalized 1-PS in $SL(2,\CC) \times SL(3,\CC)$ is of the following form:

\begin{equation} \label{moaweylchamber}
\begin{pmatrix}
t & 0 & 0 & \cdots & 0 \\
0 & t^{-1} & 0 & \cdots & 0 \\
0 & 0 & t^a & 0 & 0 \\
\vdots & \vdots & 0 & t^b & 0 \\
0 & 0 & 0 & 0 & t^c 
\end{pmatrix}
\end{equation}

where $a+b+c=0$ and $a \geq b \geq c$.  The normalization restriction of the $SL(3,\CC)$ block gives an ordering of the degree $3$ monomials in the variables $x_2,x_3,x_4$.  The weights of $x_0^2$, $x_0x_1$, and $x_2^2$ are $2$,$0$, and $-2$ respectively.  The set of maximal non-stable polynomials $F$ are those where $\mu(F,\lambda) \leq 0$ and the maximal unstable families have $\mu(F, \lambda) < 0$.  A general polynomial in (\ref{eq:moadecomp}) will be semistable if $s_3$, $t_3$, and $u_3$ have at most weights $-2$, $0$, and $2$, with respect to the action (\ref{moaweylchamber}).  These weights are neccesary in order balance the weights arising from $x_0^2$, $x_0x_1$, $x_1^2$ so that $\mu$ is less than zero.  Similary, the weights for a polynomial in (\ref{eq:moadecomp}) the weights for $s_3$, $t_3$, and $u_3$ are at most $-3$, $-1$, and $1$ for the polynomial to be unstable.  The calculation below gives the set of semistable and unstable families.  A similar method can be used for all other minimal orbits in order to explicitly calculate the set of semistable and unstable families.

\begin{figure}[htdp]
\begin{center}
   \subfloat{\label{fig:SS7-A}\includegraphics[trim =60mm 0mm 60mm 0mm, clip, scale=0.30]{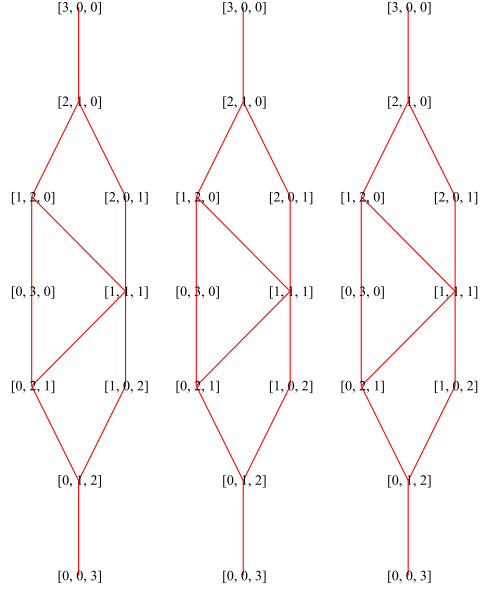}}
    \subfloat{\label{fig:SS7-B}\includegraphics[trim = 60mm 0mm 60mm 0mm, clip, scale=0.30]{family300.eps}}
    \subfloat{\label{fig:SS7-A}\includegraphics[trim = 60mm 0mm 60mm 0mm, clip, scale=0.30]{family300.eps}}
  \caption{Poset structure of equation~\ref{eq:moadecomp}}
 \end{center}
  \label{fig:SS7a}
\end{figure}

%\begin{figure}[htp]
%	\centering
%		\includegraphics[scale=0.35]{family300pdf.pdf}
%	\caption{Poset structure of degree 3 monomials}
%	\label{fig:family300}
%\end{figure}

From the GIT analysis of $MO \hyph A$, the semistable hypersurfaces and their corresponding destabilizing 1-PS $\lambda$ are given below.

\begin{enumerate}

\item [(SS1-A)] $ x_0^2( q_3(x_3,x_4)+q_1(x_2,x_3)x_4^2+x_4^3) + x_0x_1 ( q_3(x_3,x_4) + x_2x_3x_4 + q_1(x_2,x_3)x_4^2+x_3x_4^2+x_3^2x_4+x_4^3 ) + x_1^2 ( x_2q_2(x_3,x_4) + q_3(x_3,x_4) )$

\begin{enumerate}
\item[-]  Destabilizing 1-PS: $\langle 1, -1, 4, -1, -3 \rangle$
\end{enumerate}

\item [(SS2-A)]$x_0^2 ( q_3(x_3,x_4) ) + x_0x_1 ( x_2q_2(x_3,x_4) + q_3(x_3,x_4) ) + x_1^2 ( x_2q_2(x_3,x_4) + q_3(x_3,x_4) )$

\begin{enumerate}
\item[-] Destabilizing 1-PS: $\langle 1, -1, 2, -1, -1 \rangle$

\end{enumerate}

\item [(SS3-A)]$x_0^2 ( q_1(x_2,x_3)x_4^2+x_4^3 ) + x_0x_1 ( q_2(x_2,x_3)x_4 + q_1(x_2,x_3)x_4^2+ x_4^3 ) + x_1^2 ( q_2(x_2,x_3)x_4 + q_1(x_2,x_3)x_4^2+ x_4^3 )$

\begin{enumerate}
\item[-] Destabilizing 1-PS: $\langle 1, -1, 1, 1, -2 \rangle$

\end{enumerate}

\item [(SS4-A)]$x_0^2 ( x_3x_4^2+x_4^3 ) + x_0x_1 ( q_3(x_3,x_4) + x_2x_3x_4 + q_1(x_2,x_3)x_4^2+x_3x_4^2+x_3^2x_4+x_4^3 ) + x_1^2 ( x_2^2q_1(x_3,x_4) + x_2q_2(x_3,x_4) + q_3(x_3,x_4) )$

\begin{enumerate}
\item[-] Destabilizing 1-PS: $\langle 1, -1, 1, 0, -1 \rangle$

\end{enumerate}

\item [(SS5-A)]$x_0^2 ( q_1(x_2,x_3)x_4^2+x_3^2x_4+x_4^3 ) + x_0x_1 ( q_3(x_3,x_4) + x_2x_3x_4 + q_1(x_2,x_3)x_4^2+x_3x_4^2+x_3^2x_4+x_4^3 ) + x_1^2 (  x_2q_2(x_3,x_4) + q_3(x_3,x_4) + q_2(x_2,x_3)x_4 + x_4^3  )$

\begin{enumerate}
\item[-] Destabilizing 1-PS: $\langle 1, -1, 2, 0, -2 \rangle$

\end{enumerate}

\end{enumerate}

From the poset analysis, the set of unstable families are given below.

\begin{enumerate}
\item [(US1-A)] $x_0^2(q_3(x_3,x_4)+x_2x_4^2) + x_0x_1(q_3(x_3,x_4)+x_2x_4^2) + x_1^2(q_3(x_3,x_4)+x_2x_3x_4+x_2x_4^2)$ 

\item [(US2-A)] $x_0^2(q_3(x_3,x_4)) + x_0x_1(q_3(x_3,x_4)) + x_1^2(q_3(x_3,x_4)+x_2q_2(x_3,x_4))$ 

\item [(US3-A)] $x_0^2(x_2x_4^2+x_3^2x_4+x_3x_4^2+x_4^3) + x_0x_1(x_2x_4^2+x_3^2x_4+x_3x_4^2+x_4^3) + x_1^2(q_3(x_3,x_4)+x_2x_3x_4+x_2x_4^2)$ 

\item [(US4-A)] $x_0^2(x_2x_4^2+x_3x_4^2+x_4^3) + x_0x_1(x_2x_4^2+x_3^2x_4+x_3x_4^2+x_4^3) + x_1^2(q_2(x_2,x_3)x_4+q_1(x_2,x_3)x_4^2+x_4^3)$

\item [(US5-A)] $x_0^2(x_4^3) + x_0x_1(q_1(x_2,x_3)x_4^2+x_3^2x_4+x_4^3) + x_1^2(q_3(x_3,x_4)+x_2^2x_4+x_2x_3x_4+x_2x_4^2)$ 

\item [(US6-A)] $x_0^2(x_3x_4^2+x_4^3) + x_0x_1(q_1(x_2,x_3)x_4^2+x_3^2x_4+x_4^3) + x_1^2(q_3(x_3,x_4)+x_2x_3x_4+x_2x_4^2)$ 

\item [(US7-A)] $x_0^2(x_3x_4^2+x_4^3) + x_0x_1(q_1(x_2,x_3)x_4^2+x_3^2x_4+x_4^3) + x_1^2(q_2(x_2,x_3)x_4+q_1(x_2,x_3)x_4^2+x_4^3)$ 

\item [(US8-A)] $x_0^2(x_3^2x_4+x_3x_4^2+x_4^3) + x_0x_1(q_1(x_2,x_3)x_4^2+x_3^2x_4+x_4^3) + x_1^2(x_2q_2(x_3,x_4)+q_3(x_3,x_4))$ 

\item [(US9-A)] $x_0^2(x_3x_4^2+x_4^3) + x_0x_1(q_1(x_2,x_3)x_4^2+x_3^2x_4+x_4^3) + x_1^2(x_2q_2(x_3,x_4)+q_3(x_3,x_4))$ 

\item [(US10-A)] $x_0^2(x_4^3) + x_0x_1(q_1(x_2,x_3)x_4^2+x_3^2x_4+x_4^3) + x_1^2(x_2q_2(x_3,x_4)+q_3(x_3,x_4))$ 

\item [(US11-A)] $x_0^2(x_3x_4^2+x_4^3) + x_0x_1(q_1(x_2,x_3)x_4^2+x_3^2x_4+x_4^3) + x_1^2(x_2x_3x_4+x_2x_4^2+q_3(x_3,x_4))$ 

\item [(US12-A)] $x_0^2(x_3x_4^2+x_4^3) + x_0x_1(q_1(x_2,x_3)x_4^2+x_3^2x_4+x_4^3) + x_1^2(q_2(x_2,x_3)x_4+q_1(x_2,x_3)x_4^2+x_4^3)$ 

\item [(US13-A)] $x_0^2(x_4^3) + x_0x_1(q_1(x_2,x_3)x_4^2+x_3^2x_4+x_4^3) + x_1^2(q_2(x_2,x_3)x_4+q_1(x_2,x_3)x_4^2+x_4^3+x_3^3)$ 

\item [(US14-A)] $x_0^2(x_4^3+x_3^2x_4) + x_0x_1(x_4^3+x_3^2x_4) + x_1^2(x_2q_2(x_3,x_4)+q_3(x_3,x_4))$ 

\item [(US15-A)] $x_0^2(x_4^3) + x_0x_1(x_4^3+x_3^2x_4) + x_1^2(x_2q_2(x_3,x_4)+q_3(x_3,x_4)+x_2^3x_4)$ 

\item [(US16-A)] $x_0^2(x_3^2x_4+x_4^3) + x_0x_1(x_4^3+x_3^2x_4) + x_1^2(q_3(x_3,x_4)+x_2x_3x_4+x_2x_4^2)$ 

\item [(US17-A)] $x_0^2(x_3^2x_4+x_4^3) + x_0x_1(x_4^3+x_3^2x_4) + x_1^2(q_2(x_2,x_3)x_4+q_1(x_2,x_3)x_4+x_4^3)$ 

\item [(US18-A)] $x_0^2(x_4^3) + x_0x_1(x_4^3+x_3^2x_4) + x_1^2(q_3(x_3,x_4)+x_2x_3x_4+x_2x_4^2+x_3^3)$ 

\end{enumerate}

\begin{prop} Let $X$, up to a coordinate transformation,  be of the form $q_{2,3}(x_0,x_1 \parallel x_2,x_3,x_4)$.   
\begin{enumerate}
\item If $X$ belongs to one of the families  US1-A - US18-A then X is unstable. 

\item If $X$ is of type SS1-A then the orbit is not closed and it degenerates to MO2-I. 

\item If $X$ is of type SS2-A then the orbit is not closed and it degenerates to MO2-IV. 

\item If $X$ is of type SS3-A then the orbit is not closed and it degenerates to MO2-IV.

\item If $X$ is of type SS4-A then the orbit is not closed and it degenerates to  MO2-I.

\item If $X$ is of type SS5-A then the orbit is not closed and it degenerates to  M02-IV.

\end{enumerate}
Otherwise, $X$ is a closed orbit.

\end{prop}

\subsubsection{Minimal Orbit B} \label{sec:minorbitb}

In the case of $MO \hyph B$, the centralizer $Z_G(H)=\CC^*  \times SL(3,\CC) \times \CC^* \subseteq SL(5,\CC)$ acts on polynomials in $MO\hyph B$ which are of the form $q_5(x_1,x_2,x_3)+x_0x_4q_3(x_1,x_2,x_3)+x_0^2x_4^2q_1(x_1,x_2,x_3)$.  The action on $Z_G(H)$ on $x_0x_4$ and $x_0^2x_4^2$ will have weights zero.  By modifying the procedure in the $MO \hyph A$ case, one can obtain the maximal semistable and unstable families with respect to the $Z_G(H)$ action.  The poset analysis shows that the maximal semistable families are the following families below.

\begin{enumerate}

\item [(SS1-B)]$ ( x_1q_4(x_2,x_3) + q_5(x_2,x_3) ) + x_0x_4 (  x_1q_2(x_2,x_3) + q_3(x_2,x_3)  ) + x_0^2x_4^2 (  q_1(x_2,x_3) )$
\begin{enumerate}
\item[-]  Destabilizing 1-PS: $\langle 1, 2, -1, -1, -1 \rangle$

\end{enumerate}

\item [(SS2-B)]$ ( x_1q_4(x_2,x_3) + q_5(x_2,x_3) + q_2(x_1,x_2)x_3^2 ) + x_0x_4 ( q_3(x_2,x_3) + x_1x_2x_3 + q_1(x_1,x_2)x_3^2  ) + x_0^2x_4^2 (  q_1(x_2,x_3) )$

\begin{enumerate}
\item[-]  Destabilizing 1-PS:  $\langle 1, 3, -1, -2, -1 \rangle$

\end{enumerate}

\item [(SS3-B)]$ ( q_3(x_1,x_2)x_3^2 + q_2(x_1,x_2)x_3^3 + q_1(x_1,x_2)x_3^4 + x_3^5 ) + x_0x_4 (  q_2(x_1,x_2)x_3 + q_1(x_1,x_2)x_3^2 + x_3^3  ) + x_0^2x_4^2 (  x_3  )$

\begin{enumerate}
\item[-]  Destabilizing 1-PS: $\langle 1, 1, 1, -2, -1 \rangle$

\end{enumerate}

\item [(SS4-B)]$ ( q_5(x_2,x_3) + x_1x_2^3x_3 + x_1x_2^2x_3^2 + x_1x_2x_3^3 + x_1x_3^4 + x_1^2x_2x_3^2 + q_2(x_1,x_2)x_3^3  ) + x_0x_4 ( q_3(x_2,x_3) + x_1x_2x_3 + q_1(x_1,x_2)x_3^2  ) + x_0^2x_4^2 (  q_1(x_2,x_3)  )$

\begin{enumerate}
\item[-]  Destabilizing 1-PS:  $\langle 1, 1, 0, -1, -1 \rangle$

\end{enumerate}

\end{enumerate}

From the poset analysis, the set of unstable families are given below.

\begin{enumerate}
\item [(US1-B)] $(x_1q_4(x_2,x_3)+x_1^2x_3)+x_0x_4(q_3(x_2,x_3)+x_1x_3^2)+x_0^2x_4^2(x_2+x_3)$

\item [(US2-B)] $(q_5(x_2,x_3)+x_1x_2^3x_3+x_1x_2^2x_3^2+x_1^2x_3^3+x_1x_2x_3^3+x_1x_3^4)+x_0x_4(q_1(x_1,x_2)x_3^2+x_2^2x_3+x_3^3)+x_0^2x_4^2(x_3)$ 

\item [(US3-B)] $(x_1^2x_2x_3^2+x_1x_2^2x_3^2+x_2^4x_3+x_2^3x_3^2+x_2^2x_3^3+q_2(x_1,x_2)x_3^3+q_1(x_1,x_2)x_3^4+x_3^5)+x_0x_4(q_1(x_1,x_2)x_3^2+x_2^2x_3+x_3^3)+x_0^2x_4^2(x_3)$

\item [(US4-B)] $(q_3(x_1,x_2)x_3^2+q_2(x_1,x_2)x_3^3+q_1(x_1,x_2)x_3^4+x_4^5)+x_0x_4(q_1(x_1,x_2)x_3^2+x_2^2x_3+x_3^3)+x_0^2x_4^2(x_3)$ 

\end{enumerate}

\begin{prop} Let $X$, up to a coordinate transformation, be of the form  $q_5(x_1,x_2,x_3)+x_0q_3(x_1,x_2,x_3)x_4+x_0^2q_1(x_1,x_2,x_3)x_4^2$.   
\begin{enumerate}
\item If $X$ belongs to one of the families  US1-B - US4-B then $X$ is unstable. 

\item If $X$ is of type SS1-B then the orbit is not closed and it degenerates to MO2-IV. 

\item If $X$ is of type SS2-B then the orbit is not closed and it degenerates to MO2-V. 

\item If $X$ is of type SS3-B then the orbit is not closed and it degenerates to MO2-IV.

\item If $X$ is of type SS4-B then the orbit is not closed and it degenerates to  MO2-VI.

\end{enumerate}
Otherwise, $X$ is a closed orbit.

\end{prop}

\subsubsection{Minimal Orbit C} \label{sec:minorbitc}

In the case of $MO \hyph C$, the centralizer $Z_G(H)= SL(2,\CC) \times SL(2,\CC) \times \CC^* \subseteq SL(5,\CC)$ acts on polynomials in $MO\hyph C$ which are of the form $q_{1,4}(x_0,x_1 \parallel x_2,x_3) + x_4q_{2,2}(x_0,x_1 \parallel x_2,x_3) + x_4^2q_3(x_0,x_1)$.  By modifying the procedure in the $MO \hyph A$, one can obtain the maximal semistable and unstable families with respect to the $Z_G(H)$ action.  The poset analysis shows that the maximal semistable families are the following families below.

\begin{enumerate}
\item [(SS1-C)]$ ( x_0(x_2x_3^3 + x_3^4) + x_1( x_2^2x_3^2+x_2x_3^3 + x_3^4)  ) + x_4^2 ( x_0x_1^2 + x_1^3 ) + x_4 ( x_0^2 (x_3^2) + x_0x_1 ( x_2x_3 + x_3^2)  + x_1^2(x_2^2+x_2x_3 + x_3^2) )$

\begin{enumerate}
\item[-]  Destabilizing 1-PS:  $\langle 1, -1, 1, -1, 0 \rangle$

\end{enumerate}
\end{enumerate}

From the poset analysis, the set of unstable families are given below.

\begin{enumerate}
\item [(US1-C)] $x_0(x_2x_3^3+x_3^4)+x_1(x_2^2x_3^2+x_2x_3^3+x_3^3)+x_4^2(x_0x_1^2+x_1^3)+x_4x_0x_1(x_3^2)+x_4x_1^2(x_2x_3+x_3^2)$
\end{enumerate}

\begin{prop} Let $X$, up to a coordinate transformation, be of the form $q_{1,4}(x_0,x_1 \parallel x_2,x_3) + q_3(x_0,x_1)x_4^2+ q_{2,2}(x_0,x_1 \parallel x_2,x_3)x_4 	$.
  
\begin{enumerate}
\item If $X$ belongs to the family  US1-C then $X$ is unstable.

\item If $X$ is of type SS1-C then the orbit is not closed and it degenerates to MO2-VII. 

\end{enumerate}
Otherwise, $X$ is a closed orbit.

\end{prop}

\subsubsection{Minimal Orbit D} \label{sec:minorbitd}

In the case of $MO \hyph D$, the centralizer $Z_G(H)= \CC^* \times SL(4,\CC)  \subseteq SL(5,\CC)$ acts on polynomials in $MO\hyph D$ which are of the form $x_0q_4(x_1,x_2,x_3,x_4)$.  By modifying the procedure in the $MO \hyph A$ case, one can obtain the maximal semistable and unstable families with respect to the $Z_G(H)$ action.  The poset analysis shows that the maximal semistable families are the following families below.

\begin{enumerate}
\item [(SS1-D)]$ x_0 (  x_1q_3(x_2,x_3,x_4) + q_4(x_2,x_3,x_4)   )$

\begin{enumerate}
\item[-]  Destabilizing 1-PS:   $\langle 0, 3, -1, -1, -1 \rangle$
\end{enumerate}

\item [(SS2-D)]$ x_0 (  q_3(x_1,x_2,x_3)x_4 + q_2(x_1,x_2,x_3)x_4^2 + q_1(x_1,x_2,x_3)x_4^3 + x_4^4   )$

\begin{enumerate}
\item[-]  Destabilizing 1-PS:   $\langle 0, 1, 1, 1, -3 \rangle$

\end{enumerate}

\item [(SS3-D)]$ x_0 (  q_{2,2}(x_1,x_2 \parallel x_3,x_4) + q_{1,3}(x_1,x_2 \parallel x_3,x_4) + q_4(x_3,x_4)   )$
\begin{enumerate}
\item[-]  Destabilizing 1-PS:  $\langle 0, 1, 1, -1, -1 \rangle$

\end{enumerate}

\item [(SS4-D)]$ x_0 (  q_4(x_2,x_3,x_4) +   x_1q_2(x_2,x_3)x_4 + q_2(x_1,x_2,x_3)x_4^2 + q_1(x_1,x_2,x_3)x_4^3   ) $
\begin{enumerate}
\item[-]  Destabilizing 1-PS:  $\langle 0, 1, 0, 0, -1 \rangle$

\end{enumerate}

\end{enumerate}

From the poset analysis, the set of unstable families are given below.

\begin{enumerate}
\item [(US1-D)] $q_4(x_2,x_3,x_4) + q_{1,3}(x_1,x_2 \parallel x_3,x_4)+q_4(x_3,x_4)+x_1x_2x_4^2$

\item [(US2-D)] $x_2^2x_3x_4+ q_{1,3}(x_1,x_2 \parallel x_3,x_4)+q_4(x_3,x_4)+q_2(x_1,x_2,x_3)x_4^2+q_1(x_1,x_2,x_3)x_4^3$

\item [(US3-D)] $x_2q_3(x_3,x_4)+q_4(x_3,x_4)+q_3(x_2,x_3)x_4+q_2(x_2,x_3)x_4^2+q_1(x_2,x_3)x_4^3+q_1(x_1,x_2)x_3^2x_4+q_2(x_1,x_2,x_3)x_4^2+q_1(x_1,x_2,x_3)x_4^3$

\end{enumerate}

\begin{prop} Let $X$, up to a coordinate transformation, be of the form $x_0q_4(x_1,x_2,x_3,x_4)$.   

\begin{enumerate}
\item If $X$ belongs to one of the families  US1-D - US3-D then $X$ is unstable.

\item If $X$ is of type SS1-D then the orbit is not closed and it degenerates to MO2-VII. 

\item If $X$ is of type SS2-D then the orbit is not closed and it degenerates to MO2-VIII. 

\item If $X$ is of type SS3-D then the orbit is not closed and it degenerates to MO2-IX.

\item If $X$ is of type SS4-D then the orbit is not closed and it degenerates to  MO2-X.

\end{enumerate}
Otherwise, $X$ is a closed orbit.

\end{prop}

%%%%%%%%%%%%%%%%%%%%%%%%%%%%%%%%%%%%%%%%%%%%%%%%%%%%%%%%%%%%%%%%%%%%%%
\medskip
\section{Boundary Structure}
\label{sec:boundary}
\medskip

From the propositions in sections~\ref{sec:minorbita},~\ref{sec:minorbitb},~\ref{sec:minorbitc}, and~\ref{sec:minorbitd} it was shown that certain hypersurfaces in the families $MO \hyph A - MO \hyph D$ degenerate into even smaller families.  In the GIT compactification these smaller families are the boundary strata where the boundaries components $MO \hyph A - MO \hyph D$ interesect.  Again, Luna's criterion it can be determined which members of these families represent closed orbit as well as which families degenerate further.  The centralizers for the families in the second level of minimal orbits are given in table 4.

\begin{table}[htdp] \label{table:secondlevel}
\begin{center}
\renewcommand{\arraystretch}{1.1}
    \begin{tabular}{ | c | c | c | }
    \hline
    \emph{Family} & \emph{Invariant 1-PS Subgroup (H)} & \emph{Centralizer of H ($Z_G(H))$} \\  \hline
    MO2-I & $\langle 6,0,2,-3,-5 \rangle$ & Maximal Torus $T$ \\ \hline
    
   \multicolumn{3}{| p{14cm} |}{$ x_0^2 ( x_2x_4^2   ) + x_0x_1 (  x_2x_3x_4        ) + x_1^2 ( x_2x_3^2   )$     } \\ 
    \hline \hline

    MO2-II & $\langle 4,2,-1,-2,-3 \rangle$ & Maximal Torus $T$ \\ \hline
    
   \multicolumn{3}{| p{14cm} |}{$ x_0^2 ( x_3x_4^2  ) + x_0x_1 (  x_3^3 + x_2x_3x_4 ) + x_1^2 ( x_2^2x_3  )$       } \\ 
    \hline \hline

MO2-III & $\langle 4,2,0,-2,-4 \rangle$ & Maximal Torus $T$ \\ \hline
    
   \multicolumn{3}{| p{14cm} |}{$ x_0^2 ( x_2x_4^2 + x_3^2x_4  ) + x_0x_1 (  x_3^3 + x_2x_3x_4 ) + x_1^2 ( x_2x_3^2 + x_2^2x_4   )$       } \\ 
    \hline \hline

MO2-IV & $\langle 4,2,-1,-1,-5 \rangle$ & $\CC^{*2} \times SL(2,\CC) \times \CC^*$ \\ \hline
    
   \multicolumn{3}{| p{14cm} |}{$ x_0x_4 ( x_1q_2(x_2,x_3) )     $        } \\ 
    \hline \hline

MO2-V & $\langle 5,3,-1,-2,-7 \rangle$ & Maximal Torus $T$ \\ \hline
    
   \multicolumn{3}{| p{14cm} |}{$  x_0x_4 ( x_1x_2x_3 )  $        } \\ 
    \hline \hline

MO2-VI & $\langle 2,1,0,-1,-2 \rangle$ & Maximal Torus $T$ \\ \hline
    
   \multicolumn{3}{| p{14cm} |}{$ (   x_1x_2^3x_3 + x_1^2x_2x_3^2  )  + x_0x_4 (    x_1x_2x_3      )   $      } \\ 
    \hline \hline

MO2-VII & $\langle 5,3,0,-2,-6 \rangle$ & Maximal Torus $T$ \\ \hline
    
   \multicolumn{3}{| p{14cm} |}{$x_4x_0^2x_3^2 + x_4x_0x_1x_2x_3 + x_4x_1^2x_2^2  $    } \\ 
    \hline \hline

MO2-VIII & $\langle 4,2,-2,-2,-2 \rangle$ & $\CC^{*2} \times SL(3,\CC)$ \\ \hline
    
   \multicolumn{3}{| p{14cm} |}{$ x_0x_1q_3(x_2,x_3,x_4)   $     } \\ 
    \hline \hline

MO2-IX & $\langle 4,0,0,-2,-2 \rangle$ & $\CC^{*} \times SL(2,\CC) \times SL(2, \CC)$ \\ \hline
    
   \multicolumn{3}{| p{14cm} |}{$ x_0q_{2,2}(x_1,x_2 \parallel x_3,x_4)   $       } \\ 
    \hline \hline

MO2-X & $\langle 4,0,-1,-1,-2 \rangle$ & $\CC^{*2} \times SL(2,\CC) \times  \CC^* $ \\ \hline
    
   \multicolumn{3}{| p{14cm} |}{$  x_0 ( q_4(x_2,x_3) + x_1q_2(x_2,x_3)x_4 + x_1^2x_4^2 )             $   } \\ 
    \hline \hline

    \end{tabular}
	\vspace{0.1cm}
\caption{Second Level of Minimal OrbitsMO2-I - MO2-X}
\end{center}
\end{table}

\subsection{Second Level of Minimal Orbits} \label{sec:secondlevelminimalorbits}

\subsubsection{MO2-I}

In the case of family $MO2-I$, the maximal torus $T$ acts on polynomials of the form $x_0^2 (x_2x_4^2) + x_0x_1(x_2x_3x_4) + x_1^2(x_2x_3^2)$.  The weights of a 1-PS $\lambda=\langle a_0, a_1, a_2, a_3, a_4 \rangle$ of the maximal torus $T$ acting on the monomials are subject to the constraint $a_0+a_1+a_2+a_3+a_4=0$.  In a GIT semistable family all of the monomials, subject to the constraint, will have weight at most $0$.  Similarly, all GIT unstable families will have weight at most $-1$.  The set of semistable families are given below.

\begin{enumerate}
\item [SS1-I] $x_0^2 ( x_2x_4^2   ) + x_0x_1 (  x_2x_3x_4 ) $
\begin{enumerate}
\item[-]  Destabilizing 1-PS: $\langle 1,1,-1,0,-1 \rangle$
\end{enumerate}

\item [SS2-I] $ x_0x_1 (  x_2x_3x_4        ) + x_1^2 ( x_2x_3^2   )$
\begin{enumerate}
\item[-]  Destabilizing 1-PS: $\langle 1,-1,0,0,0 \rangle$
\end{enumerate}
\end{enumerate}

The set of unstable families are given below.
\begin{enumerate}
\item [US1-I] $x_0^2x_2x_4^2$
\item [US2-I]$x_1^2x_2x_3^2$
\end{enumerate}

\begin{prop} Let $X$, up to a coordinate transformation, be of the form $ x_0^2 ( x_2x_4^2   ) + x_0x_1 (  x_2x_3x_4        ) + x_1^2 ( x_2x_3^2   )$.   

\begin{enumerate}
\item If $X$ belongs to one of the families  US1-I - US2-I then $X$ is unstable.

\item If $X$ is of type SS1-I then the orbit is not closed and it degenerates to MO2-V. 

\item If $X$ is of type SS2-I then the orbit is not closed and it degenerates to MO2-V.

\end{enumerate}
Otherwise, $X$ is a closed orbit.

\end{prop}

\subsubsection{MO2-II}

In the case of family $MO2-II$, the maximal torus $T$ acts on polynomials of the form $ x_0^2 ( x_3x_4^2  ) + x_0x_1 (  x_3^3 + x_2x_3x_4 ) + x_1^2 ( x_2^2x_3  )$.  Following the same procedure as in MO2-I, the set of semistable families are given below.

\begin{enumerate}
\item [SS1-II]$x_0^2 ( x_3x_4^2  ) + x_0x_1 (  x_3^3 + x_2x_3x_4 ) $
\begin{enumerate}
\item[-]  Destabilizing 1-PS: $\langle 1,1,0,-1,-1 \rangle$
\end{enumerate}

\item [SS2-II] $x_0x_1 (  x_3^3 + x_2x_3x_4 ) + x_1^2 ( x_2^2x_3  )$

\begin{enumerate}
\item \textbf{1-PS} $\langle 1,0,0,-1,0 \rangle$
\end{enumerate}

\item [SS3-II] $ x_0x_1 (  x_2x_3x_4 ) + x_1^2 ( x_2^2x_3  )$

\begin{enumerate}
\item[-]  Destabilizing 1-PS: $\langle 1,-1,0,0,0 \rangle$
\end{enumerate}

\item [SS4-II]$x_0^2 ( x_3x_4^2  ) + x_0x_1 (   x_2x_3x_4 ) $
\begin{enumerate}
\item[-]  Destabilizing 1-PS: $\langle 1,1,0,-1,-1 \rangle$
\end{enumerate}

\item [SS5-II] $ x_0x_1 (  x_3^3 + x_2x_3x_4 ) $

\begin{enumerate}
\item[-]  Destabilizing 1-PS: $\langle 1,-2,0,0,1 \rangle$
\end{enumerate}

\item [SS6-II]$ x_0x_1 ( x_2x_3x_4 )$

\begin{enumerate}
\item[-]  Destabilizing 1-PS: No 1-PS
\end{enumerate}

\end{enumerate}

The set of unstable families are given below.
\subsubsection{Unstable Family}
\begin{enumerate}
\item [US1-II] $x_0^2x_3x_4^2+x_0x_1x_3^2$
\item [US2-II]$x_0x_1x_3^2+x_1^2x_2^2x_3$
\item [US3-II]$x_0x_1x_3^2$
\item [US4-II]$x_0^2x_3x_4^2$
\item [US5-II]$x_1^2x_2^2x_3$
\end{enumerate}

\begin{prop} Let $X$, up to a coordinate transformation, be of the form $ x_0^2 ( x_3x_4^2  ) + x_0x_1 (  x_3^3 + x_2x_3x_4 ) + x_1^2 ( x_2^2x_3  )$ .   

\begin{enumerate}
\item If $X$ belongs to one of the families  US1-II - US5-II then $X$ is unstable.
\item If $X$ is of type SS1-II then the orbit is not closed and it degenerates to MO2-V. 

\item If $X$ is of type SS2-II then the orbit is not closed and it degenerates to MO2-V. 

\item If $X$ is of type SS3-II then the orbit is not closed and it degenerates to MO2-V. 

\item If $X$ is of type SS4-III then the orbit is not closed and it degenerates to MO2-V. 

\item If $X$ is of type SS5-II then the orbit is not closed and it degenerates to MO2-V. 

\item If $X$ is of type SS6-II then the orbit is not closed and it degenerates to MO2-V.

\end{enumerate}
Otherwise, $X$ is a closed orbit.

\end{prop}

\subsubsection{MO2-III}

In the case of family $MO2-III$, the maximal torus $T$ acts on polynomials of the form $ x_0^2 ( x_2x_4^2 + x_3^2x_4  ) + x_0x_1 (  x_3^3 + x_2x_3x_4 ) + x_1^2 ( x_2x_3^2 + x_2^2x_4   )$.  Following the same procedure as in MO2-I, the set of semistable families are given below.

\begin{enumerate}
\item [SS1-III]$ x_0^2 ( x_2x_4^2 + x_3^2x_4  ) + x_0x_1 (  x_3^3 + x_2x_3x_4 ) $

\begin{enumerate}
\item[-]  Destabilizing 1-PS: $\langle 1,3,3,-4,-3 \rangle$
\end{enumerate}

\item [SS2-III]$  x_0x_1 (  x_3^3 + x_2x_3x_4 ) + x_1^2 ( x_2x_3^2 + x_2^2x_4 )$

\begin{enumerate}
\item[-]  Destabilizing 1-PS: $\langle 1,-8,1,2,4 \rangle$
\end{enumerate}

\item [SS3-III] $ x_0^2 ( x_2x_4^2 + x_3^2x_4 ) + x_0x_1 (  x_2x_3x_4 ) $
\begin{enumerate}
\item[-]  Destabilizing 1-PS: $\langle 1,2,0,0,-3 \rangle$
\end{enumerate}

\item [SS4-III]$  x_0x_1 (   x_2x_3x_4 ) + x_1^2 ( x_2x_3^2 + x_2^2x_4 )$
\begin{enumerate}
\item[-]  Destabilizing 1-PS: $\langle 1,-1,0,0,0 \rangle$
\end{enumerate}

\item [SS5-III]$ x_0^2 (  x_3^2x_4 ) + x_0x_1 (  x_3^3 + x_2x_3x_4 ) + x_1^2 ( x_2x_3^2 )$
\begin{enumerate}
\item[-]  Destabilizing 1-PS: $\langle 1,1,2,-7,3 \rangle$
\end{enumerate}

\item [SS6-III] $ x_0^2 (  x_3^2x_4  ) + x_0x_1 ( x_2x_3x_4 ) + x_1^2 ( x_2x_3^2 )$

\begin{enumerate}
\item[-]  Destabilizing 1-PS: $\langle 1,0,0,-2,1 \rangle$
\end{enumerate}

\item [SS7-III]$ x_0^2 ( x_2x_4^2 ) + x_0x_1 ( x_2x_3x_4 ) + x_1^2 ( x_2^2x_4 )$

\begin{enumerate}
\item[-]  Destabilizing 1-PS: $\langle 1,0,-1,1,-1 \rangle$
\end{enumerate}

\item [SS8-III]$ x_0^2 ( x_2x_4^2 ) + x_0x_1 (  x_3^3 + x_2x_3x_4 ) $

\begin{enumerate}
\item[-]  Destabilizing 1-PS: $\langle 1,1,1,-1,-2 \rangle$
\end{enumerate}

\item [SS9-III]$ x_0^2 ( x_3^2x_4 ) + x_0x_1 ( x_3^3 + x_2x_3x_4 ) $

\begin{enumerate}
\item[-]  Destabilizing 1-PS: $\langle 1,0,0,-2,1 \rangle$
\end{enumerate}

\item [SS10-III]$  x_0x_1 ( x_3^3 + x_2x_3x_4 ) + x_1^2 (  x_2^2x_4 )$
\begin{enumerate}
\item[-]  Destabilizing 1-PS: $\langle 1,-2,0,0,1 \rangle$
\end{enumerate}

\item [SS11-III]$  x_0x_1 ( x_3^3 + x_2x_3x_4 ) + x_1^2 ( x_2x_3^2 )$

\begin{enumerate}
\item[-]  Destabilizing 1-PS: $\langle 1,-2,0,0,1 \rangle$
\end{enumerate}

\item [SS12-III] $  x_0x_1 ( x_2x_3x_4 ) + x_1^2 ( x_2^2x_4 )$

\begin{enumerate}
\item[-]  Destabilizing 1-PS: $\langle 1,0,0,0,-1 \rangle$
\end{enumerate}

\item [SS13-III]$  x_0x_1 (  x_2x_3x_4 ) + x_1^2 ( x_2x_3^2 )$

\begin{enumerate}
\item \textbf{1-PS} $\langle 1,-1,0,0,0 \rangle$
\end{enumerate}

\item [SS14-III]$ x_0^2 ( x_2x_4^2 ) + x_0x_1 ( x_2x_3x_4 ) $

\begin{enumerate}
\item[-]  Destabilizing 1-PS: $\langle 1,1,-1,0,-1 \rangle$
\end{enumerate}

\item [SS15-III]$ x_0^2 (  x_3^2x_4 ) + x_0x_1 (  x_2x_3x_4 ) $

\begin{enumerate}
\item[-]  Destabilizing 1-PS: $\langle 1,2,0,0,-3 \rangle$
\end{enumerate}

\item [SS16-III]$  x_0x_1 (x_3^3 + x_2x_3x_4 ) $

\begin{enumerate}
\item[-]  Destabilizing 1-PS: $\langle 1,-2,0,0,1 \rangle$
\end{enumerate}

\item [SS17-III]$  x_0x_1 (   x_2x_3x_4 )$

\begin{enumerate}
\item[-]  Destabilizing 1-PS: No 1-PS
\end{enumerate}

\end{enumerate}

The set of unstable families are given below.
\begin{enumerate}
\item [US1-III]$ x_0^2 \bigg( x_2x_4^2 + x_3^2x_4  \bigg) + x_0x_1 \bigg(  x_3^3  \bigg)  $

\item [US2-III]$  x_0x_1 \bigg(  x_3^3 +  \bigg) + x_1^2 \bigg( x_2x_3^2 + x_2^2x_4   \bigg)$ 

\item [US3-III] $ x_0^2 \bigg(  x_3^2x_4  \bigg) + x_0x_1 \bigg(  x_3^3  \bigg) + x_1^2 \bigg( x_2x_3^2   \bigg)$ 

\item [US4-III]$ x_0^2 \bigg( x_2x_4^2   \bigg) + x_0x_1 \bigg(  x_3^3  \bigg) + x_1^2 \bigg(  x_2^2x_4   \bigg)$

\item [US5-III] $ x_0^2 \bigg( x_2x_4^2 + x_3^2x_4  \bigg)   $

\item [US6-III] $   x_1^2 \bigg( x_2x_3^2 + x_2^2x_4   \bigg)$ 

\item [US7-III] $ x_0^2 \bigg(  x_3^2x_4  \bigg)  + x_1^2 \bigg( x_2x_3^2   \bigg)$ 

\item [US8-III] $ x_0^2 \bigg( x_2x_4^2   \bigg)  + x_1^2 \bigg(  x_2^2x_4   \bigg)$

\item [US9-III] $ x_0^2 \bigg( x_2x_4^2   \bigg) + x_0x_1 \bigg(  x_3^3  \bigg)  $

\item [US10-III] $  x_0x_1 \bigg(  x_3^3 +  \bigg) + x_1^2 \bigg(  x_2^2x_4   \bigg)$

\item [US11-III] $ x_0^2 \bigg(  x_3^2x_4  \bigg) + x_0x_1 \bigg(  x_3^3  \bigg)  $

\item [US12-III] $  x_0x_1 \bigg(  x_3^3 +  \bigg) + x_1^2 \bigg( x_2x_3^2   \bigg)$

\item [US13-III] $ x_0^2 \bigg( x_2x_4^2  \bigg) $

\item [US14-III] $ x_0^2 \bigg( x_3^2x_4  \bigg) $

\item [US15-III] $  x_0x_1 \bigg(  x_3^3 \bigg) $

\item [US16-III] $  x_1^2 \bigg( x_2x_3^2  \bigg)$                              

\item [US17-III] $x_1^2 \bigg( x_2^2x_4   \bigg)$

\end{enumerate}

\begin{prop} Let $X$, up to a coordinate transformation, be of the form $ x_0^2 ( x_2x_4^2 + x_3^2x_4  ) + x_0x_1 (  x_3^3 + x_2x_3x_4 ) + x_1^2 ( x_2x_3^2 + x_2^2x_4   )$  .   

\begin{enumerate}
\item If $X$ belongs to one of the families  US1-III - US17-III then $X$ is unstable.

\item If $X$ is of type SS1-III then the orbit is not closed and it degenerates to MO2-V. 

\item If $X$ is of type SS2-III then the orbit is not closed and it degenerates to MO2-V. 
\item If $X$ is of type SS3-III then the orbit is not closed and it degenerates to MO2-V. 
\item If $X$ is of type SS4-III then the orbit is not closed and it degenerates to MO2-V. 
\item If $X$ is of type SS5-III then the orbit is not closed and it degenerates to MO2-V. 
\item If $X$ is of type SS6-III then the orbit is not closed and it degenerates to MO2-V. 
\item If $X$ is of type SS7-III then the orbit is not closed and it degenerates to MO2-V. 
\item If $X$ is of type SS8-III then the orbit is not closed and it degenerates to MO2-V. 
\item If $X$ is of type SS9-III then the orbit is not closed and it degenerates to MO2-V. 
\item If $X$ is of type SS10-III then the orbit is not closed and it degenerates to MO2-V. 
\item If $X$ is of type SS11-III then the orbit is not closed and it degenerates to MO2-V. 
\item If $X$ is of type SS12-III then the orbit is not closed and it degenerates to MO2-V. 
\item If $X$ is of type SS13-III then the orbit is not closed and it degenerates to MO2-V. 
\item If $X$ is of type SS14-III then the orbit is not closed and it degenerates to MO2-V. 
\item If $X$ is of type SS15-III then the orbit is not closed and it degenerates to MO2-V. 
\item If $X$ is of type SS16-III then the orbit is not closed and it degenerates to MO2-V. 
\item If $X$ is of type SS17-III then the orbit is not closed and it degenerates to MO2-V.

\end{enumerate}
Otherwise, $X$ is a closed orbit.

\end{prop}

\subsubsection{MO2-IV}

In the case of family $MO2-IV$, the centralizer $\CC^{*2} \times SL(2,\CC) \times \CC^*$ on polynomials of the form $ x_0x_4 ( x_1q_2(x_2,x_3) ) $.  By following a similar procedure as in the $MO \hyph A$ case one can obtain the maximal semistable and unstable families with respect to the $Z_G(H)$ action.  The poset analysis shows that the maximal semistable families are the following families below.

\begin{enumerate}
\item [SS1-IV] $ x_0x_4 ( x_1x_2x_3 + x_1x_3^2 )     $
\begin{enumerate}
\item[-]  Destabilizing 1-PS: $\langle 1,0,1,0,-2 \rangle$
\end{enumerate}

\end{enumerate}

The set of unstable families are given below.

\begin{enumerate}
\item [US1-IV] $x_0x_4x_1x_3^2$

\end{enumerate}

\begin{prop} Let $X$, up to a coordinate transformation, be of the form $ x_0x_4 ( x_1q_2(x_2,x_3) )     $  .   

\begin{enumerate}
\item If $X$ belongs to one of the families  US1-IV then $X$ is unstable.
\item If $X$ is of type SS1-IV then the orbit is not closed and it degenerates to MO2-V.

\end{enumerate}
Otherwise, $X$ is a closed orbit.

\end{prop}

\subsubsection{MO2-V}

In the case of family $MO2-V$, the maximal torus $T$ acts on polynomials of the form $  x_0x_4 ( x_1x_2x_3 )$.  Following the same procedure as in MO2-I, the set of semistable families are given below.

\begin{enumerate}
\item [SS1-V] $ x_0x_4 x_2x_3x_1     $

\begin{enumerate}
\item \textbf{1-PS} No 1-PS
\end{enumerate}

\end{enumerate}

The set of unstable families are given below.
\begin{enumerate}
\item NONE

\end{enumerate}

\begin{prop} Let $X$, up to a coordinate transformation, be of the form $  x_0x_4  x_1x_2x_3   $ then it is a closed orbit.

\end{prop}

\subsubsection{MO2-VI}

In the case of family $MO2-VI$, the maximal torus $T$ acts on polynomials of the form $  ( x_1x_2^3x_3 + x_1^2x_2x_3^2) + x_0x_4 ( x_1x_2x_3  )$  .  Following the same procedure as in MO2-I, the set of semistable families are given below.

\begin{enumerate}
\item [SS1-VI]$ (   x_1x_2^3x_3 + x_1^2x_2x_3^2 )  + x_0x_4 ( x_1x_2x_3 )   $
\begin{enumerate}
\item[-]  Destabilizing 1-PS: $\langle 1,0,-1,0,0 \rangle$
\end{enumerate}

\item [SS2-VI]$ ( x_1x_2^3x_3 )  + x_0x_4 ( x_1x_2x_3 )   $
\begin{enumerate}
\item[-]  Destabilizing 1-PS: $\langle 1,0,-1,0,0 \rangle$
\end{enumerate}

\item [SS3-VI]$ (  x_1^2x_2x_3^2 )  + x_0x_4 ( x_1x_2x_3 )   $

\begin{enumerate}
\item[-]  Destabilizing 1-PS: $\langle 1,0,-1,0,0 \rangle$
\end{enumerate}

\item [SS4-VI] $x_0x_4 (    x_1x_2x_3 )   $
\begin{enumerate}
\item[-]  Destabilizing 1-PS: No 1-PS
\end{enumerate}

\end{enumerate}

The set of unstable families are given below.

\begin{enumerate}
\item [US1-VI] $(  x_1x_2^3x_3 + x_1^2x_2x_3^2 )$
\item [US2-VI] $( x_1x_2^3x_3 )$
\item [US3-VI] $( x_1^2x_2x_3^2 )$

\end{enumerate}

\begin{prop} Let $X$, up to a coordinate transformation, be of the form $ (   x_1x_2^3x_3 + x_1^2x_2x_3^2  )  + x_0x_4 (    x_1x_2x_3      )   $    .   

\begin{enumerate}
\item If $X$ belongs to one of the families  US1-VI - US3-VI then $X$ is unstable. 
\item If $X$ is of type SS1-VI then the orbit is not closed and it degenerates to MO2-V. 

\item If $X$ is of type SS2-VI then the orbit is not closed and it degenerates to MO2-V.

\item If $X$ is of type SS3-VI then the orbit is not closed and it degenerates to MO2-V. 

\item If $X$ is of type SS4-VI then the orbit is not closed and it degenerates to MO2-V.

\end{enumerate}
Otherwise, $X$ is a closed orbit.

\end{prop}

\subsubsection{MO2-VII}

In the case of family $MO2-VII$, the maximal torus $T$ acts on polynomials of the form $ x_4x_0^2x_3^2 + x_4x_0x_1x_2x_3 + x_4x_1^2x_2^2  $.  Following the same procedure as in MO2-I, the set of semistable families are given below.

\begin{enumerate}
\item [SS1-VII] $ x_4x_0^2x_3^2 + x_4x_0x_1x_2x_3   $
\begin{enumerate}
\item[-]  Destabilizing 1-PS: $\langle 1,2,0,0,0 \rangle$
\end{enumerate}

\item [SS2-VII] $  x_4x_0x_1x_2x_3 + x_4x_1^2x_2^2   $
\begin{enumerate}
\item[-]  Destabilizing 1-PS: $\langle 1,0,0,0,-1 \rangle$
\end{enumerate}

\item [SS3-VII]$  x_4x_0x_1x_2x_3   $
\begin{enumerate}
\item[-]  Destabilizing 1-PS: No 1-PS
\end{enumerate}

\end{enumerate}

The set of unstable families are given below.
\begin{enumerate}
\item [US1-VII] $x_4x_1^2x_2^2$
\item [US2-VII] $x_4x_0^2x_3^2$
\end{enumerate}

\begin{prop} Let $X$, up to a coordinate transformation, be of the form $x_4x_0^2x_3^2 + x_4x_0x_1x_2x_3 + x_4x_1^2x_2^2  $  .   

\begin{enumerate}
\item If $X$ belongs to one of the families  US1-VII - US2-VII then $X$ is unstable. 

\item If $X$ is of type SS1-VII then the orbit is not closed and it degenerates to MO2-V. 

\item If $X$ is of type SS2-VII then the orbit is not closed and it degenerates to MO2-V. 
\item If $X$ is of type SS3-VII then the orbit is not closed and it degenerates to MO2-V.

\end{enumerate}
Otherwise, $X$ is a closed orbit.

\end{prop}

\subsubsection{MO2-VIII}

In the case of family $MO2-VIII$, the centralizer $\CC^{*2} \times SL(3,\CC)$ on polynomials of the form $x_0x_1q_3(x_2,x_3,x_4) $.  By following a similar procedure as in the $MO \hyph A$ case one can obtain the maximal semistable and unstable families with respect to the $Z_G(H)$ action.  The poset analysis shows that the maximal semistable families are the following families below.

\begin{enumerate}

\item  [SS1-VIII]$ x_0x_1( x_2q_2(x_3,x_4) + q_3(x_3,x_4))   $          
\begin{enumerate}
\item[-]  Destabilizing 1-PS: $\langle 1,-1,2,-1,-1 \rangle$
\end{enumerate}

\item [SS2-VIII] $ x_0x_1( q_2(x_2,x_3)x_4 + q_1(x_2,x_3)x_4^2 + x_4^3 )   $          
\begin{enumerate}
\item[-]  Destabilizing 1-PS: $\langle 1,-1,1,1,-2 \rangle$
\end{enumerate}

\end{enumerate}

The set of unstable families are given below.
\begin{enumerate}
\item [US1-VIII] $x_0x_1 \bigg( q_3(x_3,x_4)+x_2x_4^2 \bigg)$
\end{enumerate}

\begin{prop} Let $X$, up to a coordinate transformation, be of the form $ x_0x_1q_3(x_2,x_3,x_4)   $    .   

\begin{enumerate}
\item If $X$ belongs to one of the families  US1-VIII then $X$ is unstable.
\item If $X$ is of type SS1-VIII then the orbit is not closed and it degenerates to MO2-IV.
 
\item If $X$ is of type SS2-VIII then the orbit is not closed and it degenerates to MO2-IV.

\end{enumerate}
Otherwise, $X$ is a closed orbit.

\end{prop}

\subsubsection{MO2-IX}

In the case of family $MO2-IX$, the centralizer $\CC^{*} \times SL(2,\CC) \times SL(2, \CC)$ on polynomials of the form $ x_0q_{2,2}(x_1,x_2 \parallel x_3,x_4)  $.  By following a similar procedure as in the $MO \hyph A$ case one can obtain the maximal semistable and unstable families with respect to the $Z_G(H)$ action.  The poset analysis shows that the maximal semistable families are the following families below.

\begin{enumerate}
\item [SS1-IX]$x_0 ( x_1x_2 + x_2^2 \parallel x_3x_4 + x_4^2 )$

\begin{enumerate}
\item[-]  Destabilizing 1-PS: $\langle 0,1,-1,1,-1 \rangle$
\end{enumerate}

\end{enumerate}

The set of unstable families are given below.
\begin{enumerate}
\item [US1-IX] $x_0 \bigg(  x_2^2 \parallel x_3x_4 + x_4^2 \bigg)$
\item [US2-IX] $x_0 \bigg( x_1x_2 + x_2^2 \parallel  x_4^2 \bigg)$
\item [US3-IX] $x_0 \bigg(  x_2^2 \parallel  x_4^2 \bigg)$

\end{enumerate}

\begin{prop} Let $X$, up to a coordinate transformation, be of the form $ x_0q_{2,2}(x_1,x_2 \parallel x_3,x_4)$.   

\begin{enumerate}
\item If $X$ belongs to one of the families  US1-IX - US3-IX then $X$ is unstable.
\item If $X$ is of type SS1-IX then the orbit is not closed and it degenerates to MO2-V.

\end{enumerate}
Otherwise, $X$ is a closed orbit.

\end{prop}

\subsubsection{MO2-X}

In the case of family $MO2-X$, the centralizer $\CC^{*2} \times SL(2,\CC) \times  \CC^* $ on polynomials of the form $  x_0 ( q_4(x_2,x_3) + x_1q_2(x_2,x_3)x_4 + x_1^2x_4^2 ) $ .  By following a similar procedure as in the $MO \hyph A$ case one can obtain the maximal semistable and unstable families with respect to the $Z_G(H)$ action.  The poset analysis shows that the maximal semistable families are the following families below.

\begin{enumerate}

\item [SS1-X]$  x_0 ( q_4(x_2,x_3) + x_1 (  x_2x_3 + x_3^2  ) x_4 + x_1^2x_4^2 )   $       
\begin{enumerate}
\item[-]  Destabilizing 1-PS: $\langle a,b,1,-1,c \rangle$
\end{enumerate}

\item [SS2-X]$  x_0 ( q_4(x_2,x_3) + x_1 (x_2x_3 + x_2^2  ) x_4 + x_1^2x_4^2 ) $       
\begin{enumerate}
\item[-]  Destabilizing 1-PS: $\langle a,b,1,-1,c \rangle$
\end{enumerate}

\end{enumerate}

The set of unstable families are given below.
\begin{enumerate}
\item [US1-X] $ x_0 \bigg( x_2^2x_3^2+x_2x_3^3+x_3^4 + x_1x_3^2x_4 + x_1^2x_4^2 \bigg) $
\item [US2-X]  $ x_0 \bigg( x_2x_3^3+x_3^4 + x_1(x_2x_3+x_3^2)x_4 + x_1^2x_4^2 \bigg) $

\end{enumerate}

\begin{prop} Let $X$, up to a coordinate transformation, be of the form $  x_0 ( q_4(x_2,x_3) + x_1q_2(x_2,x_3)x_4 + x_1^2x_4^2 )             $  .   

\begin{enumerate}
\item If $X$ belongs to one of the families  US1-X - US2-X then $X$ is unstable. 

\item If $X$ is of type SS1-X then the orbit is not closed and it degenerates to MO2-V. 

\item If $X$ is of type SS2-X then the orbit is not closed and it degenerates to MO2-V.

\end{enumerate}
Otherwise, $X$ is a closed orbit.

\end{prop}

\subsection{Boundary Stratification}
From the results above, the most degenerate point in the GIT compactification is the normal crossing singularities hypersurface $x_0x_1x_2x_3x_4$. The following chart show how the various degenerations occur on the boundary of the GIT compactification.  The non-closed orbits of families $MO\hyph A - MO \hyph D$ degenerate further into the families $MO2 \hyph I - MO2 \hyph X$, which all eventually degenerate to the family of $x_0x_1x_2x_3x_4$.

\pagestyle{empty}

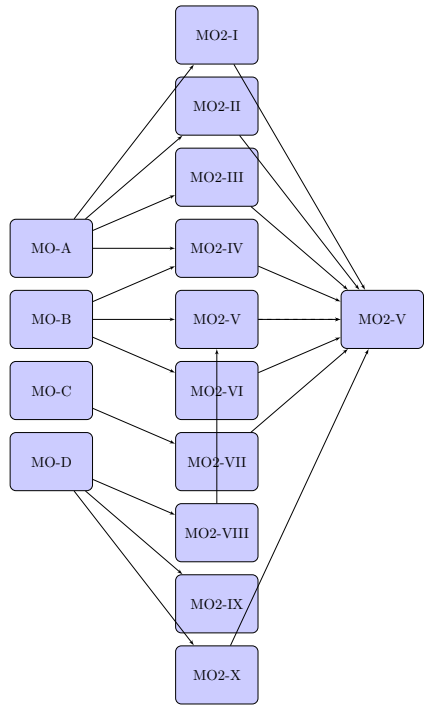
\begin{figure}[htdp]
\scalebox{0.55}{
% Define block styles
\tikzstyle{block} = [rectangle, draw, fill=blue!20, 
    text width=5em, text centered, rounded corners, minimum height=4em]
\tikzstyle{line} = [draw, -latex']
    
\begin{tikzpicture}[node distance = 1.72cm, auto]
    % Place nodes
    \node [block] (MO-A) {MO-A};
    \node [block] (MO-B) [below of=MO-A]  {MO-B};

    \node [block] (MO-C) [below of=MO-B]  {MO-C};
    \node [block] (MO-D) [below of=MO-C]  {MO-D};

    \node [block] (MO2-IV) [right of=MO-A, node distance = 4cm]  {MO2-IV};
    \node [block] (MO2-V) [below of=MO2-IV]  {MO2-V};
    \node [block] (MO2-VI) [below of=MO2-V]  {MO2-VI};
    \node [block] (MO2-VII) [below of=MO2-VI]  {MO2-VII};

    \node [block] (MO2-VIII) [below of=MO2-VII]  {MO2-VIII};
    \node [block] (MO2-IX) [below of=MO2-VIII]  {MO2-IX};
    \node [block] (MO2-X) [below of=MO2-IX]  {MO2-X};
    \node [block] (MO2-III) [above of=MO2-IV]  {MO2-III};
    \node [block] (MO2-II) [above of=MO2-III]  {MO2-II};
    \node [block] (MO2-I) [above of=MO2-II]  {MO2-I};
    
    \node [block] (MO2-V2) [right of=MO2-V, node distance=4cm]  {MO2-V};

		\path[line] (MO-A) -- (MO2-I);
		\path[line] (MO-A) -- (MO2-IV);
		\path[line] (MO-A) -- (MO2-II);
		\path[line] (MO-A) -- (MO2-III);
		\path[line] (MO-B) -- (MO2-IV);
		\path[line] (MO-B) -- (MO2-V);
		\path[line] (MO-B) -- (MO2-VI);
		\path[line] (MO-C) -- (MO2-VII);
		\path[line] (MO-D) -- (MO2-VIII);
		\path[line] (MO-D) -- (MO2-IX);
		\path[line] (MO-D) -- (MO2-X);
		
		\path[line] (MO2-I) -- (MO2-V2);
    \path[line] (MO2-II) -- (MO2-V2);
    \path[line] (MO2-III) -- (MO2-V2);
    \path[line] (MO2-IV) -- (MO2-V2);
    \path[line] (MO2-V) -- (MO2-V2);
    \path[line] (MO2-VI) -- (MO2-V2);
    \path[line] (MO2-VII) -- (MO2-V2);
    \path[line] (MO2-VIII) -- (MO2-V);
    \path[line] (MO2-X) -- (MO2-V2);
    
    \path [line, dashed] (MO2-V) -- (MO2-V2);

\end{tikzpicture}
}
\caption{Boundary Stratification of GIT Compactification}
\end{figure}
\newpage

\newpage
\appendix

\section{Code for Poset Structure}
\label{appendix:posetcode}
\hspace{ 4 in}

\pagestyle{empty}
\DefineParaStyle{Maple Heading 1}
\DefineParaStyle{Maple Text Output}
\DefineParaStyle{Maple Dash Item}
\DefineParaStyle{Maple Bullet Item}
\DefineParaStyle{Maple Normal}
\DefineParaStyle{Maple Heading 4}
\DefineParaStyle{Maple Heading 3}
\DefineParaStyle{Maple Heading 2}
\DefineParaStyle{Maple Warning}
\DefineParaStyle{Maple Title}
\DefineParaStyle{Maple Error}
\DefineCharStyle{Maple Hyperlink}
\DefineCharStyle{Maple 2D Math}
\DefineCharStyle{Maple Maple Input}
\DefineCharStyle{Maple 2D Output}
\DefineCharStyle{Maple 2D Input}
\begin{maplegroup}
\begin{mapleinput}
\mapleinline{active}{1d}{withposets();

poset := proc (n) local P, z1, z2, z3, z4, z5, z6, x0, x1, x2, x3, x4, x5, y1, y2, y3, y4, y5, y6; 

P := NULL; 
for z1 from 0 to n do 
for z2 from 0 to n do 
for z3 from 0 to n do 
for z4 from 0 to n do 
for z5 from 0 to n do 

for y1 from 0 to n do 
for y2 from 0 to n do 
for y3 from 0 to n do 
for y4 from 0 to n do 
for y5 from 0 to n do 

}{}
\end{mapleinput}
\end{maplegroup}
\begin{maplegroup}
\begin{mapleinput}
\mapleinline{active}{1d}{if z1+z2+z3+z4+z5 = n and y1+y2+y3+y4+y5 = n and z1 <= y1 and z1+z2 <= y1+y2 and z1+z2+z3 <= y1+y2+y3 and z1+z2+z3+z4 <= y1+y2+y3+y4 and x0\symbol{94}z1*x1\symbol{94}z2*x2\symbol{94}z3*x3\symbol{94}z4*x4\symbol{94}z5 <> x0\symbol{94}y1*x1\symbol{94}y2*x2\symbol{94}y3*x3\symbol{94}y4*x4\symbol{94}y5 then
 
P := P, [[z1, z2, z3, z4, z5], [y1, y2, y3, y4, y5]] 
else P := P 
end if 
end do 
end do 
end do 
end do 
end do 
end do 
end do 
end do 
end do 
end do; 
P := covers(\{P\}) 
end;}{}
\end{mapleinput}
\end{maplegroup}

\section{Sample Linear Programming Calculation}
\label{appendix:linearprogram}
\hspace{ 4 in}

\pagestyle{empty}
\DefineParaStyle{Maple Heading 1}
\DefineParaStyle{Maple Text Output}
\DefineParaStyle{Maple Dash Item}
\DefineParaStyle{Maple Bullet Item}
\DefineParaStyle{Maple Normal}
\DefineParaStyle{Maple Heading 4}
\DefineParaStyle{Maple Heading 3}
\DefineParaStyle{Maple Heading 2}
\DefineParaStyle{Maple Warning}
\DefineParaStyle{Maple Title}
\DefineParaStyle{Maple Error}
\DefineCharStyle{Maple Hyperlink}
\DefineCharStyle{Maple 2D Math}
\DefineCharStyle{Maple Maple Input}
\DefineCharStyle{Maple 2D Output}
\DefineCharStyle{Maple 2D Input}
\begin{maplegroup}
\begin{mapleinput}
\mapleinline{active}{1d}{with(Optimization);
}{}
\end{mapleinput}
\end{maplegroup}
\begin{maplegroup}
\begin{mapleinput}
\mapleinline{active}{1d}{with(LinearAlgebra);}{}
\end{mapleinput}
\end{maplegroup}
\begin{maplegroup}
\begin{mapleinput}
\mapleinline{active}{1d}{with(VectorCalculus);
}{}
\end{mapleinput}
\end{maplegroup}
\begin{maplegroup}
\begin{mapleinput}
\mapleinline{active}{1d}{with(ListTools);
}{}
\end{mapleinput}
\end{maplegroup}
\begin{maplegroup}
\begin{mapleinput}
\mapleinline{active}{1d}{v1:=[a,b,c,d,e];}{}
\end{mapleinput}
\end{maplegroup}
\begin{maplegroup}
\begin{mapleinput}
\mapleinline{active}{1d}{v2:=[4,1,0,0,0];
}{}
\end{mapleinput}
\end{maplegroup}
\begin{maplegroup}
\begin{mapleinput}
\mapleinline{active}{1d}{mon:=DotProduct(v1,v2);}{}
\end{mapleinput}
\end{maplegroup}
\begin{maplegroup}
\begin{mapleinput}
\mapleinline{active}{1d}{constraints:=\{mon<=0,a+b+c+d+e=0,a>=1,a>=b,b>=c,c>=d,d>=e\};}{}
\end{mapleinput}
\end{maplegroup}
\begin{maplegroup}
\begin{mapleinput}
\mapleinline{active}{1d}{LPSolve(1,constraints,assume=integer);
}{}
\end{mapleinput}
\end{maplegroup}

\newpage

\begin{figure}[htdp]
	\centering
		\includegraphics[scale=0.85]{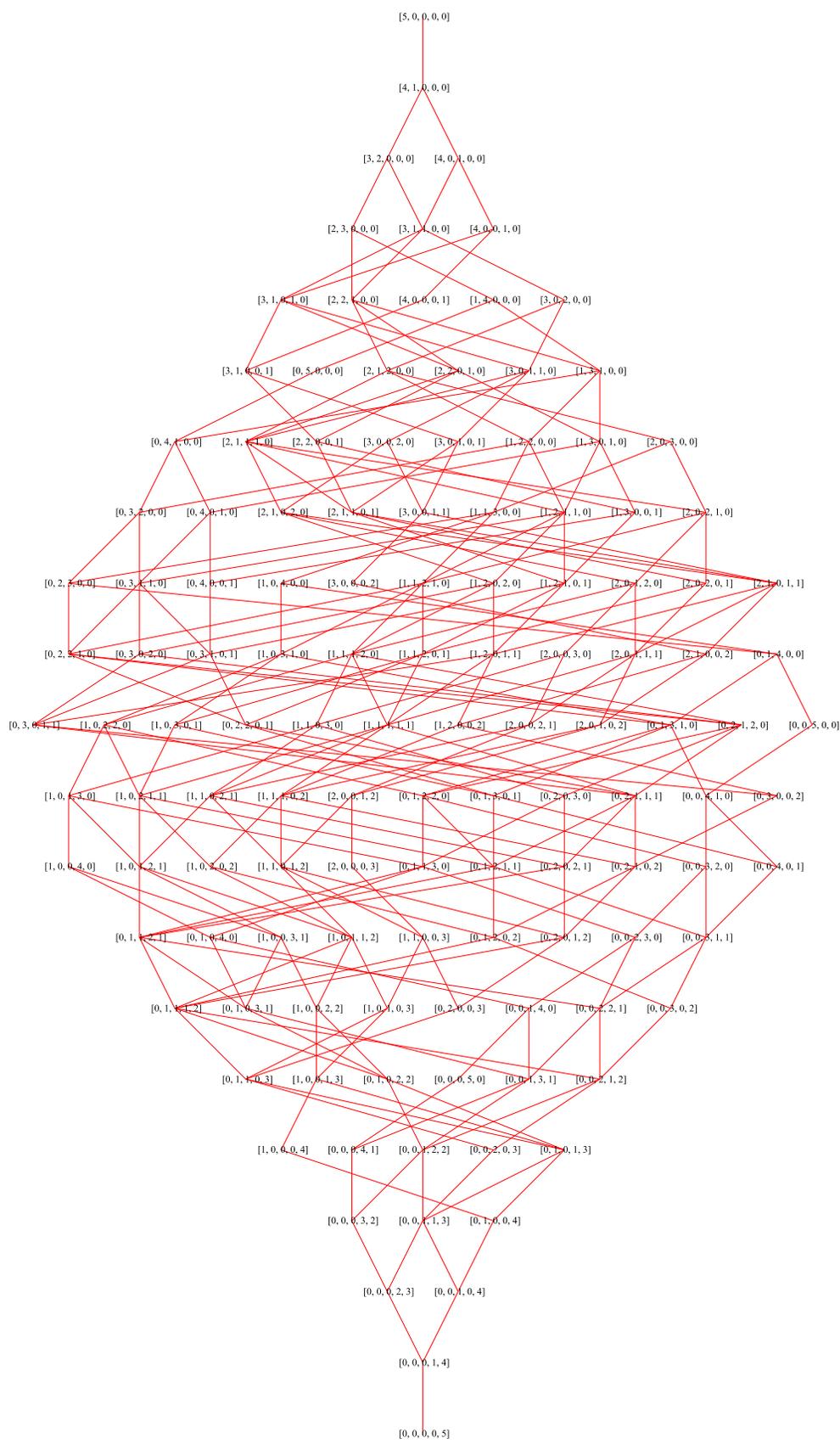}
	\caption{Poset structure of quintic monomials}
	\label{fig:calabiyau}
\end{figure}

\newpage

\begin{figure}[htdp]
 \centering
		\includegraphics[scale=.80]{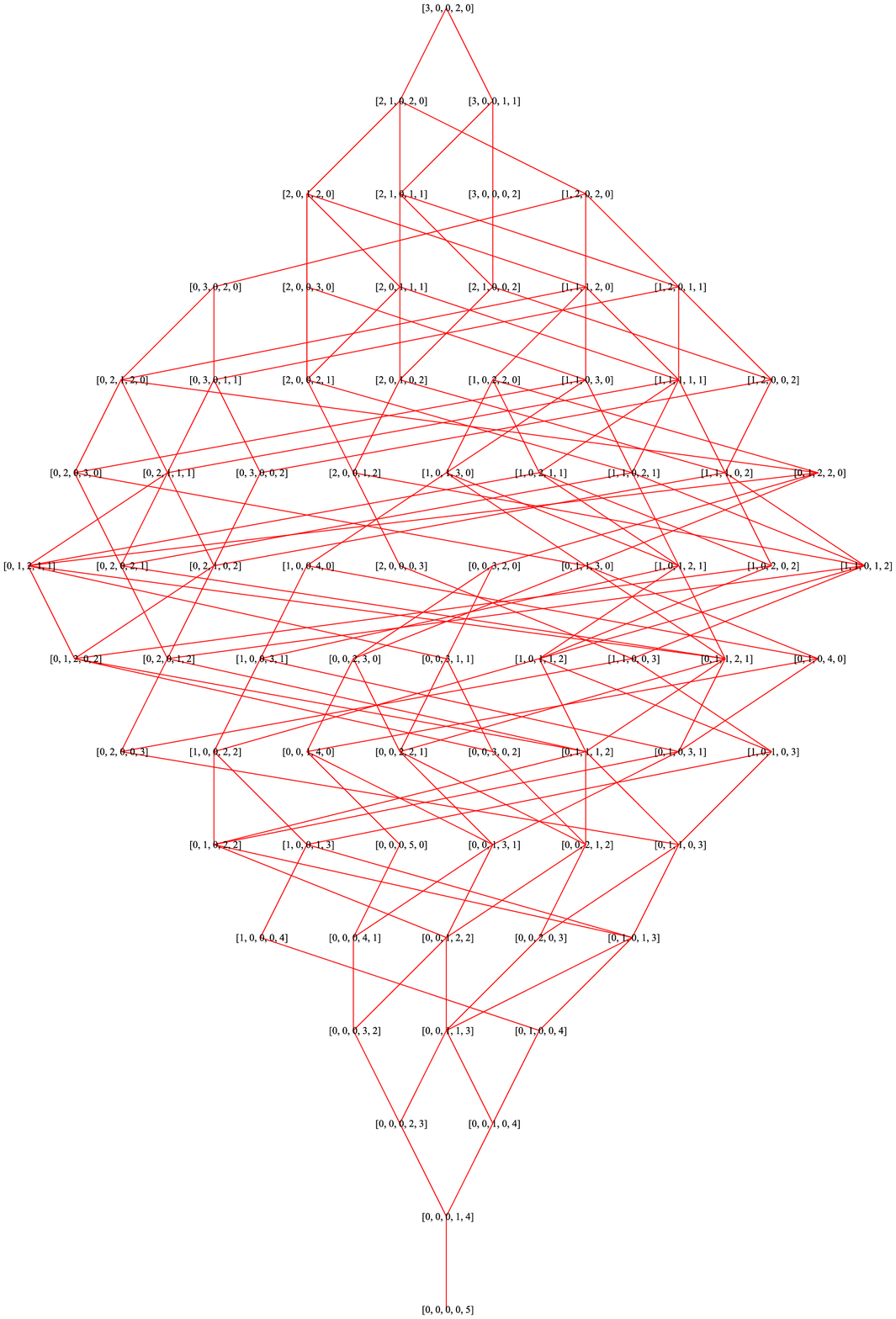}
	\caption{Poset structure of family SS1 }
	\label{fig:SS1}
\end{figure}

\newpage
\begin{figure}[htdp]
	\centering
		\includegraphics[scale=.80]{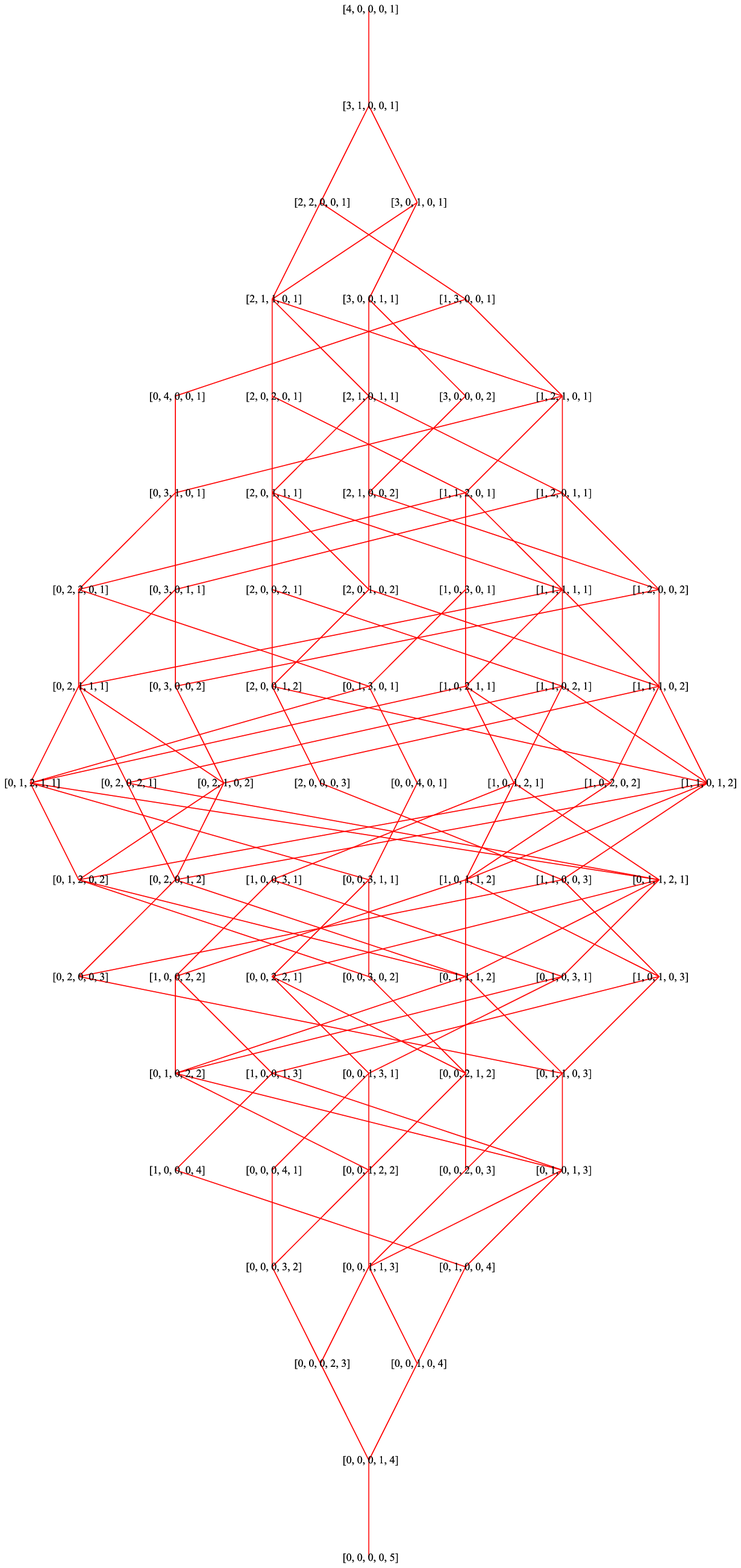}
	\caption{Poset structure of family SS2 }
	\label{fig:SS2}
\end{figure}

\newpage

\begin{figure}[htdp]
	\centering
		\includegraphics[scale=.90]{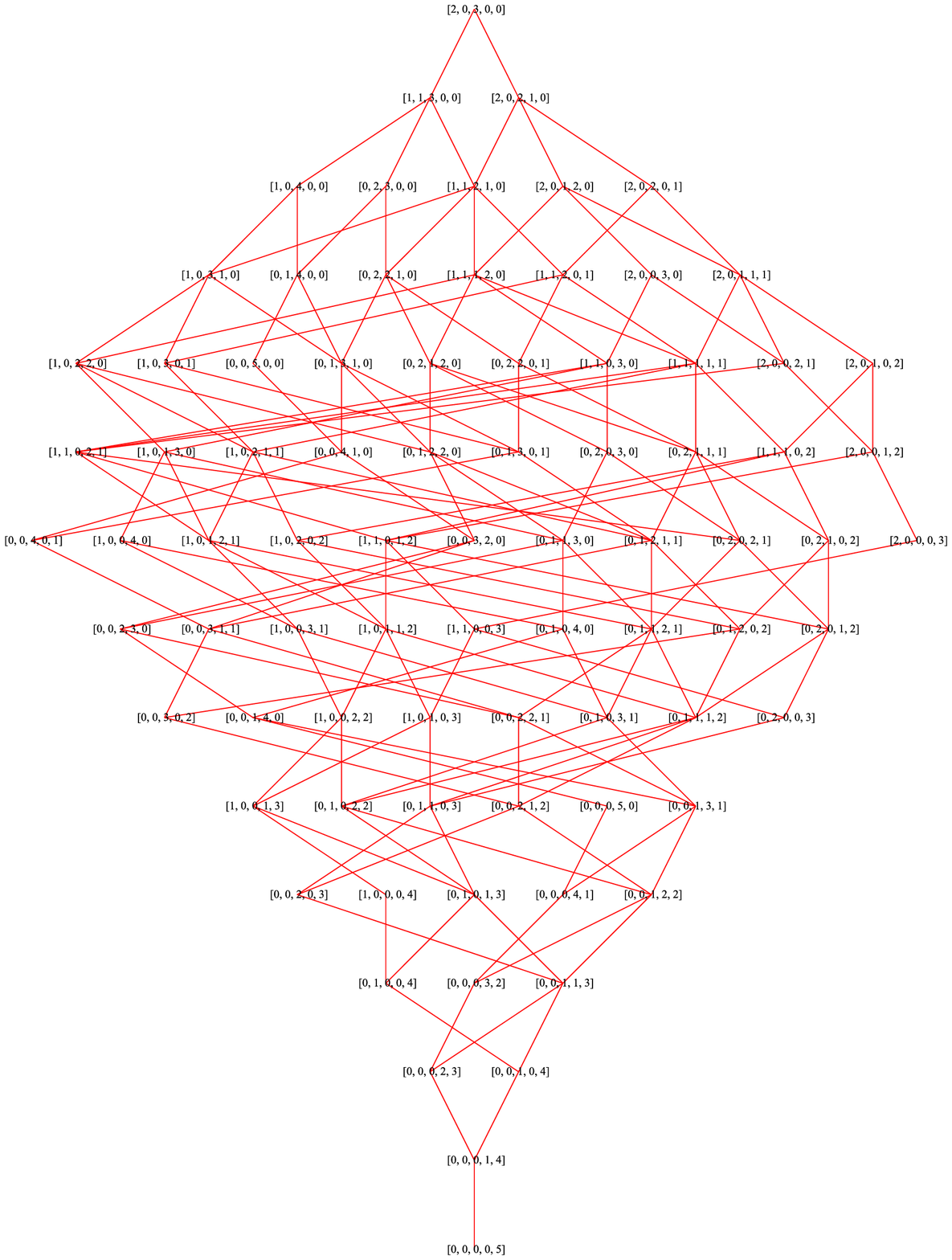}
	\caption{Poset structure of family SS3 }
	\label{fig:SS3}
\end{figure}

\newpage
\begin{figure}[htdp] 
	\centering
		\includegraphics[scale=.90]{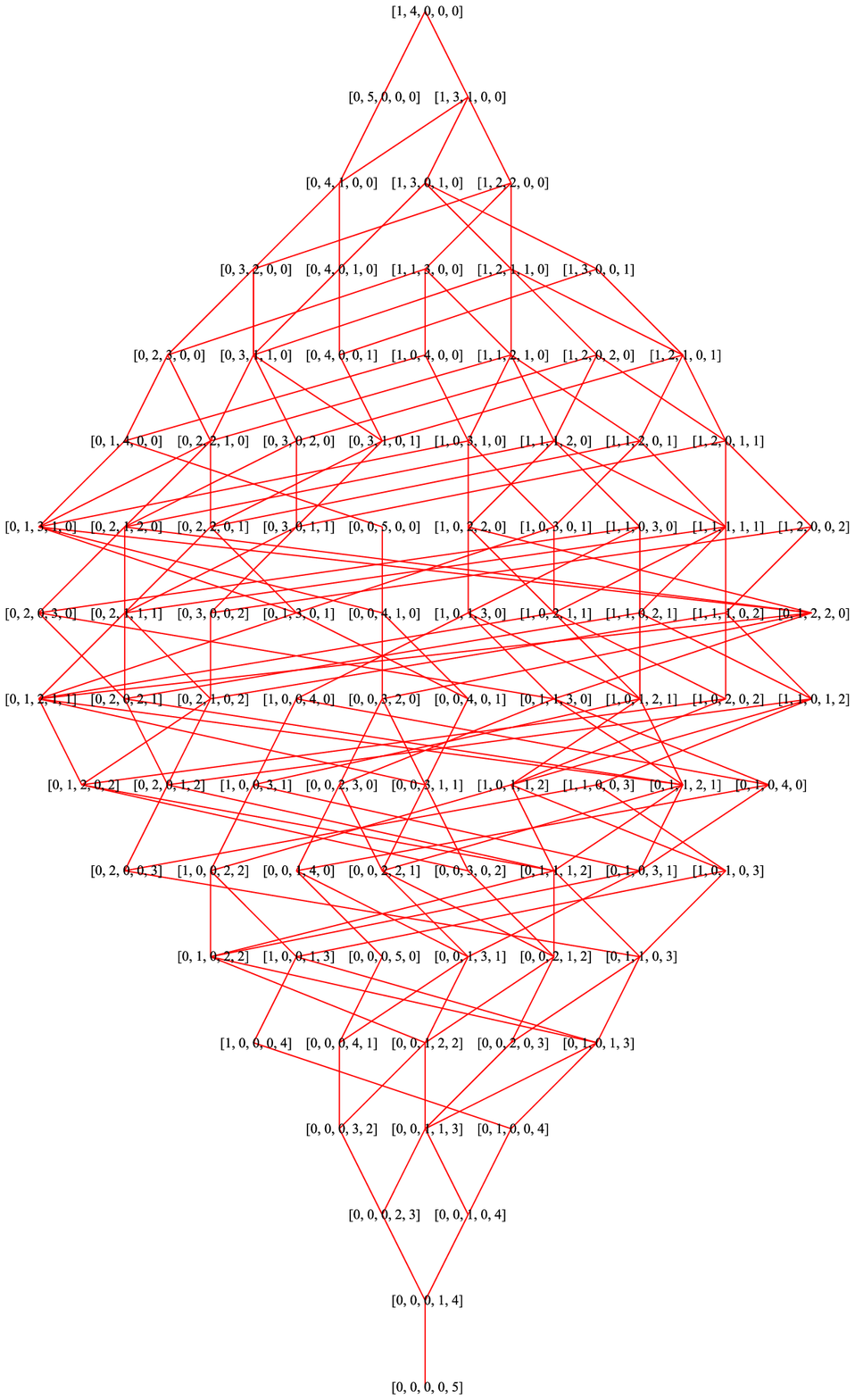}
	\caption{Poset structure of family SS4 }
	\label{fig:SS4}
\end{figure}

\newpage

\begin{figure}[htdp]   
\centering
    \subfloat { \includegraphics[trim = 50mm 0mm 50mm 0mm, clip, scale=0.85]{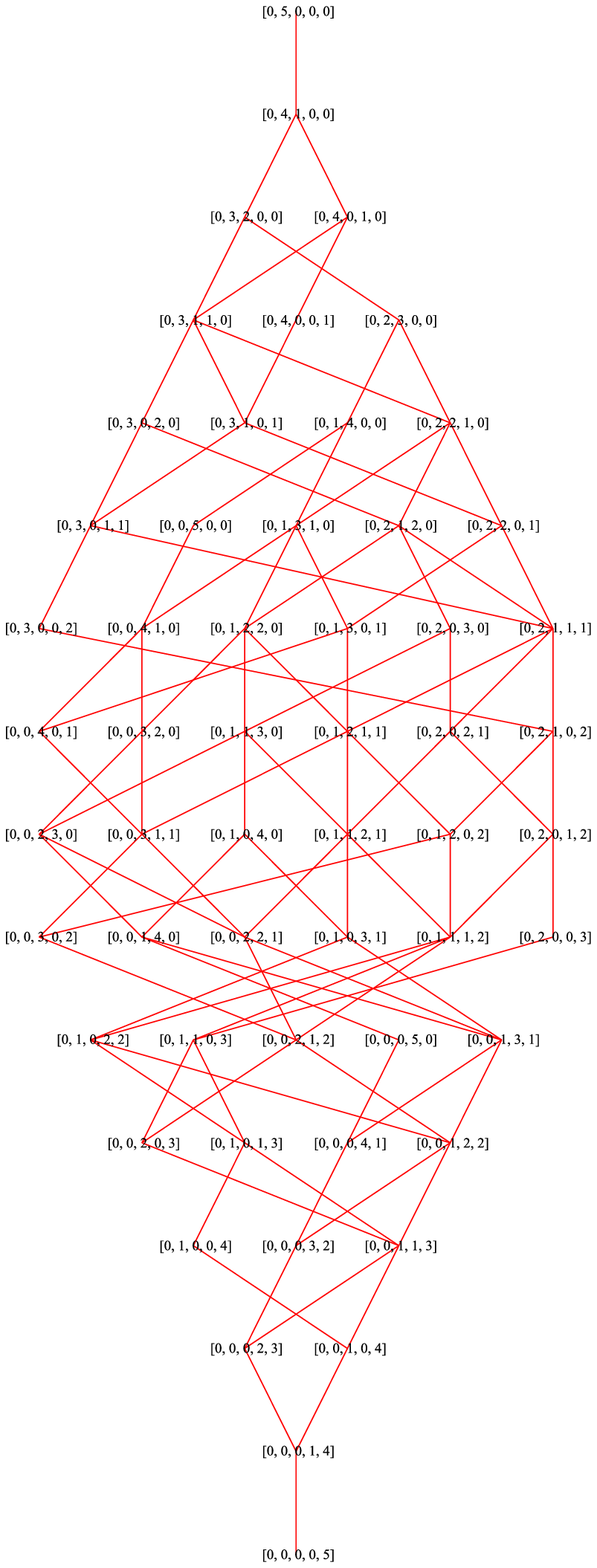}}
    \subfloat { \includegraphics[trim = 50mm 0mm 50mm 0mm, clip, scale=0.85]{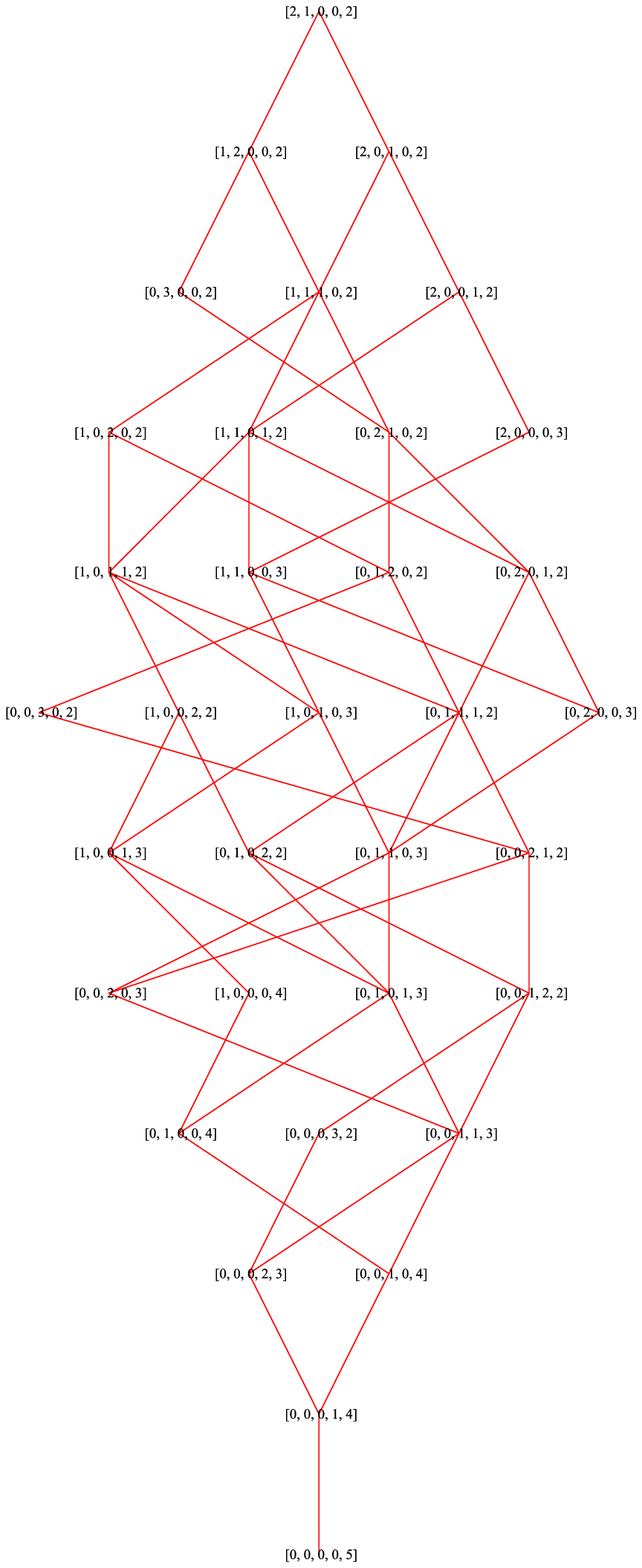}}
  \caption{Poset structure of family SS5}
\label{fig:SS5a}
\end{figure}

\newpage

\begin{figure}[htdp] 
	\centering
		\includegraphics[scale=.95]{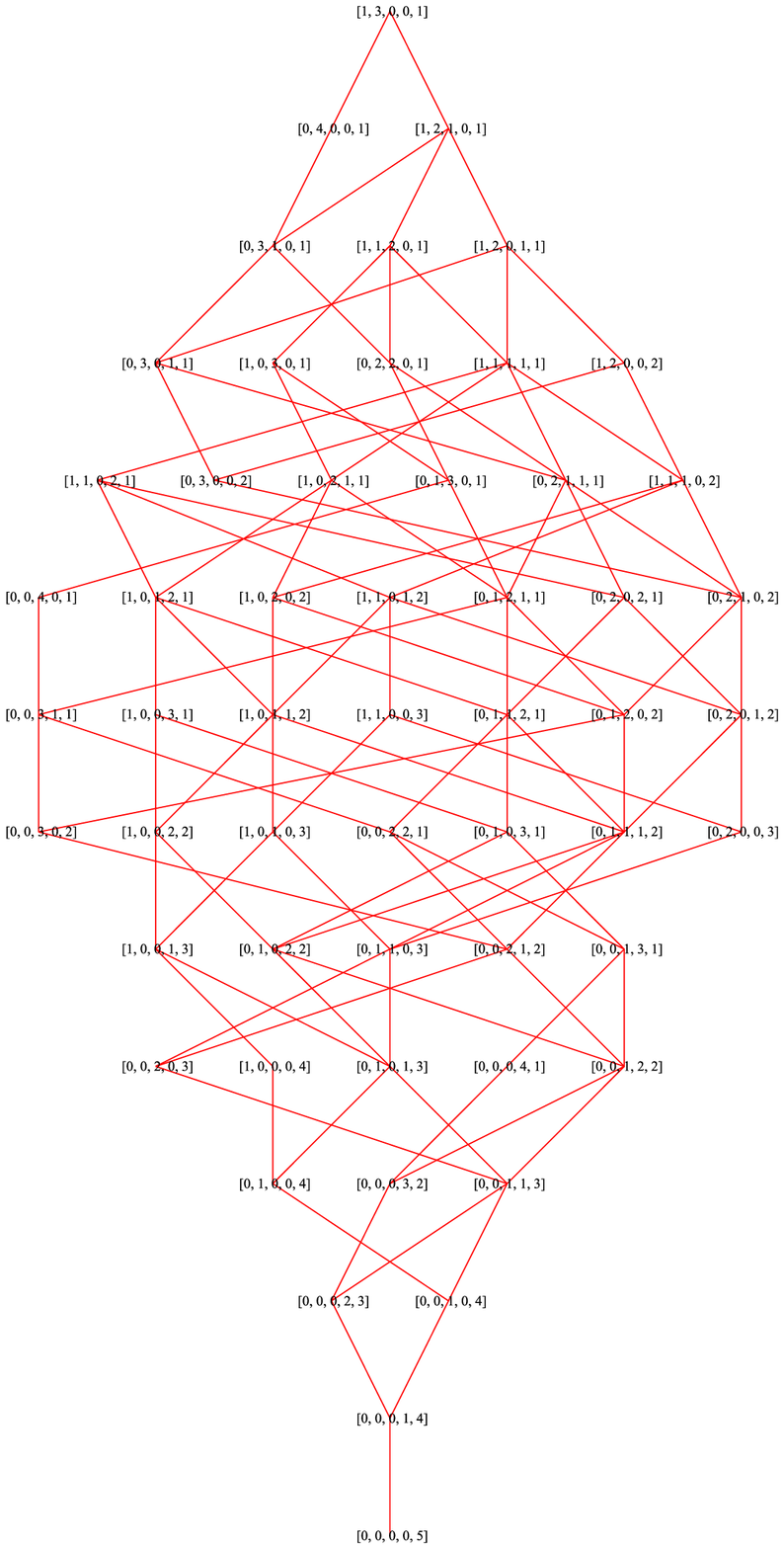}
	\caption{Poset structure of family SS5 }
	\label{fig:SS5b}
\end{figure}

\newpage

\begin{figure}[htdp]
\centering
    \subfloat{\includegraphics[trim = 40mm 0mm 32mm 0mm, clip, scale=0.75]{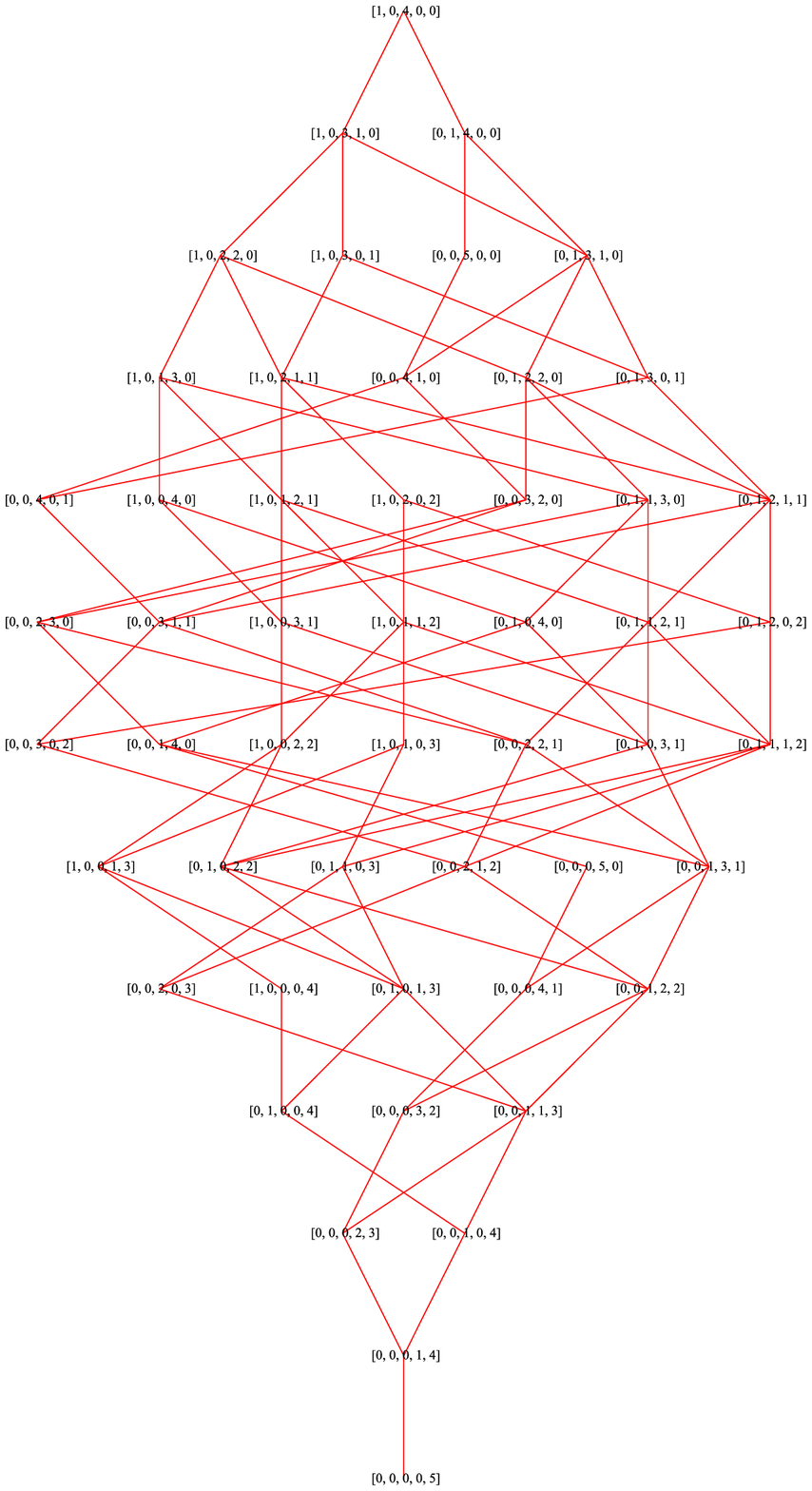}}
    \subfloat{\includegraphics[trim = 60mm 0mm 60mm 0mm, clip, scale=0.75]{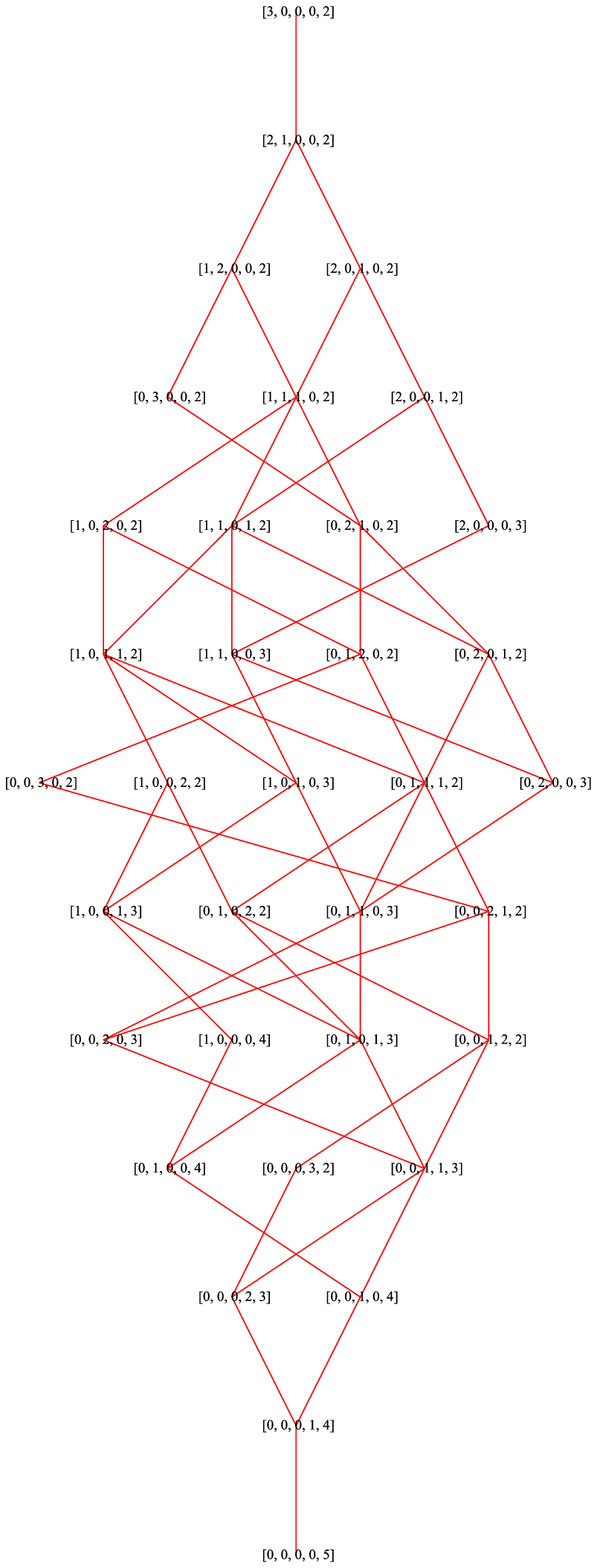}}
  \caption{Poset structure of family SS6}
  \label{fig:SS6a}
\end{figure}

\newpage

\begin{figure}[htdp] 
	\centering
		\includegraphics[scale=.85]{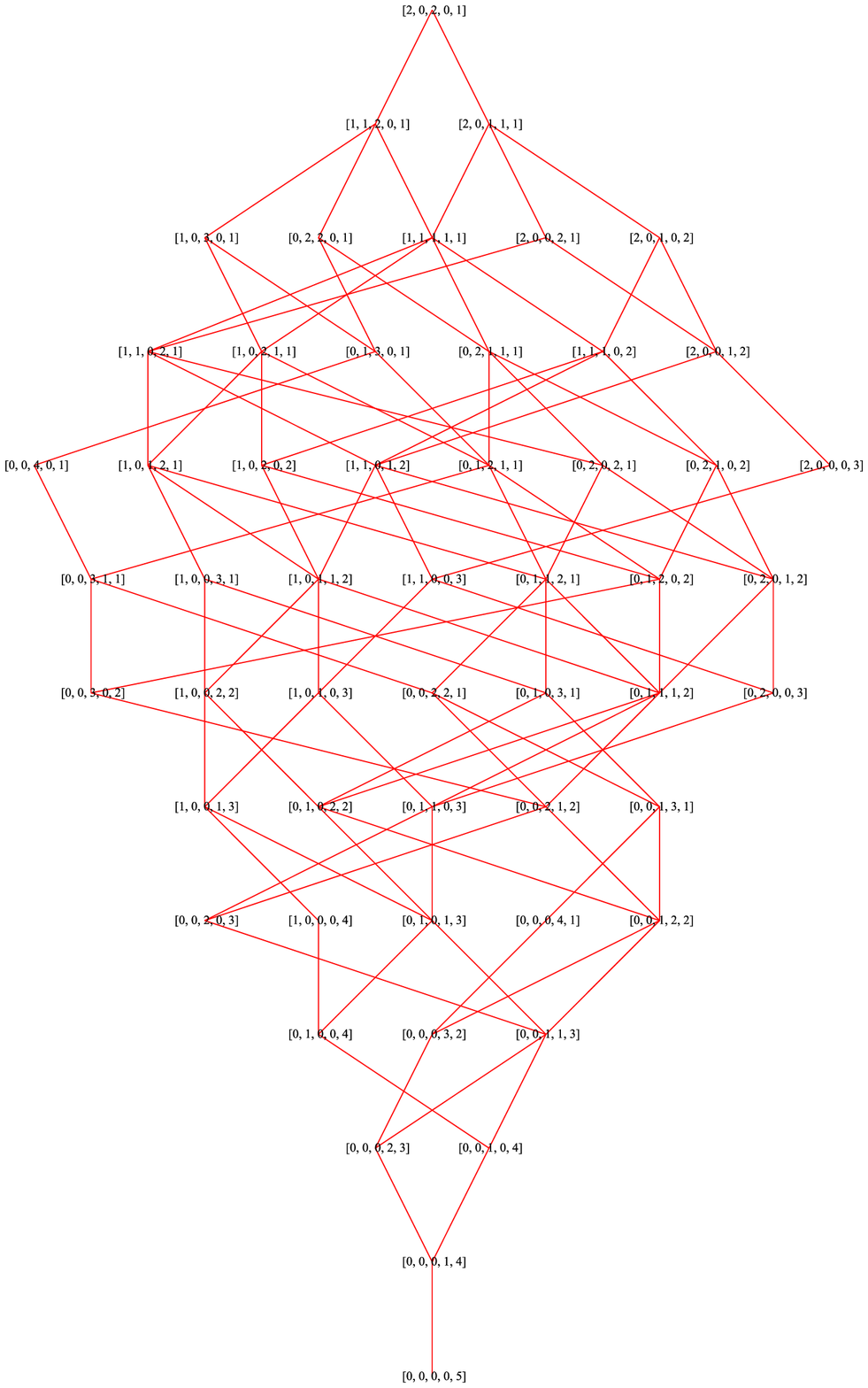}
	\caption{Poset structure of family SS6 }
	\label{fig:SS6b}
\end{figure}

\newpage
\begin{figure}[htdp]
\centering
    \subfloat{\includegraphics[trim = 52mm 0mm 50mm 0mm, clip, scale=0.75]{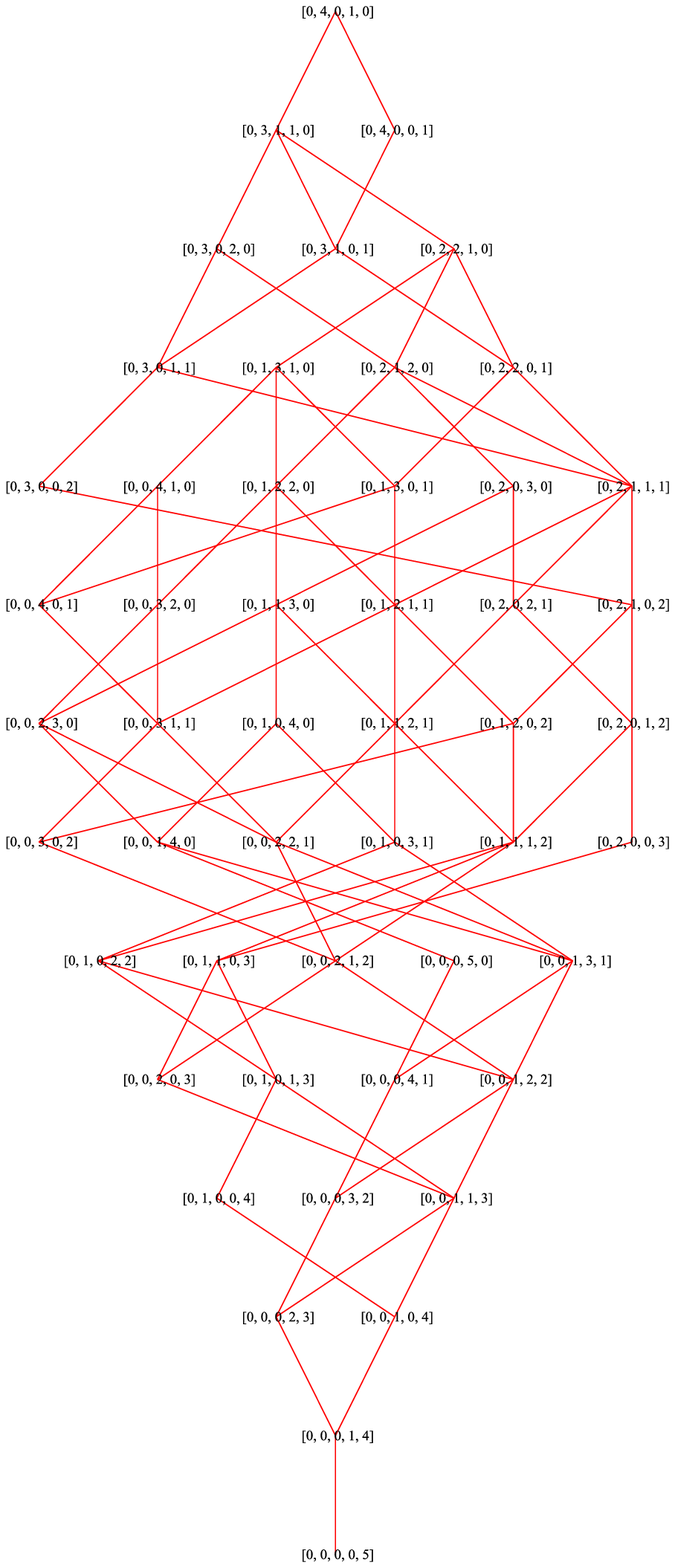}}
    \subfloat{\includegraphics[trim = 36mm 0mm 0mm 0mm, clip, scale=0.75]{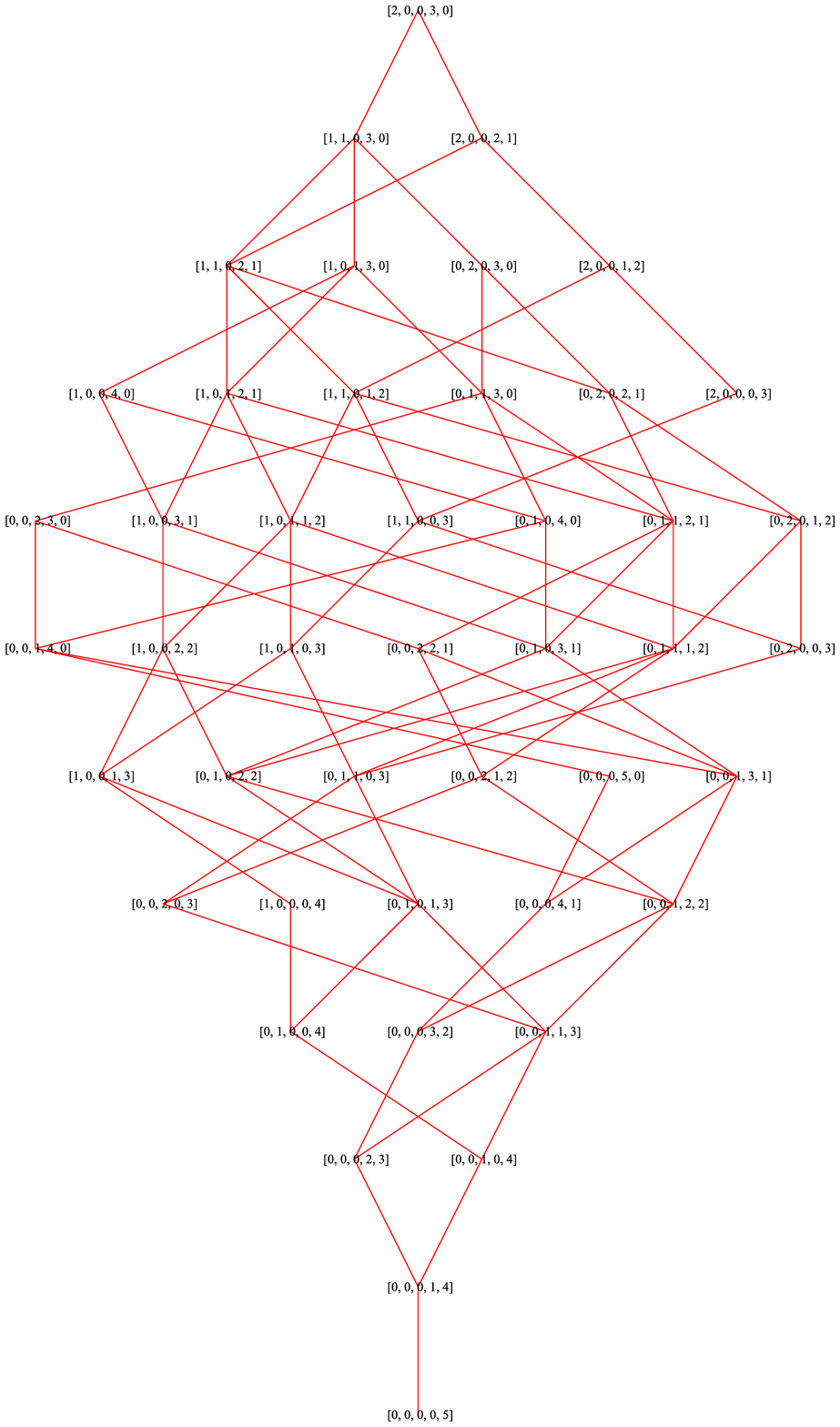}}
  \caption{Poset structure of family SS7}
  \label{fig:SS7a}
\end{figure}

\newpage

\begin{figure}[h] 
	\centering
		\includegraphics[scale=.85]{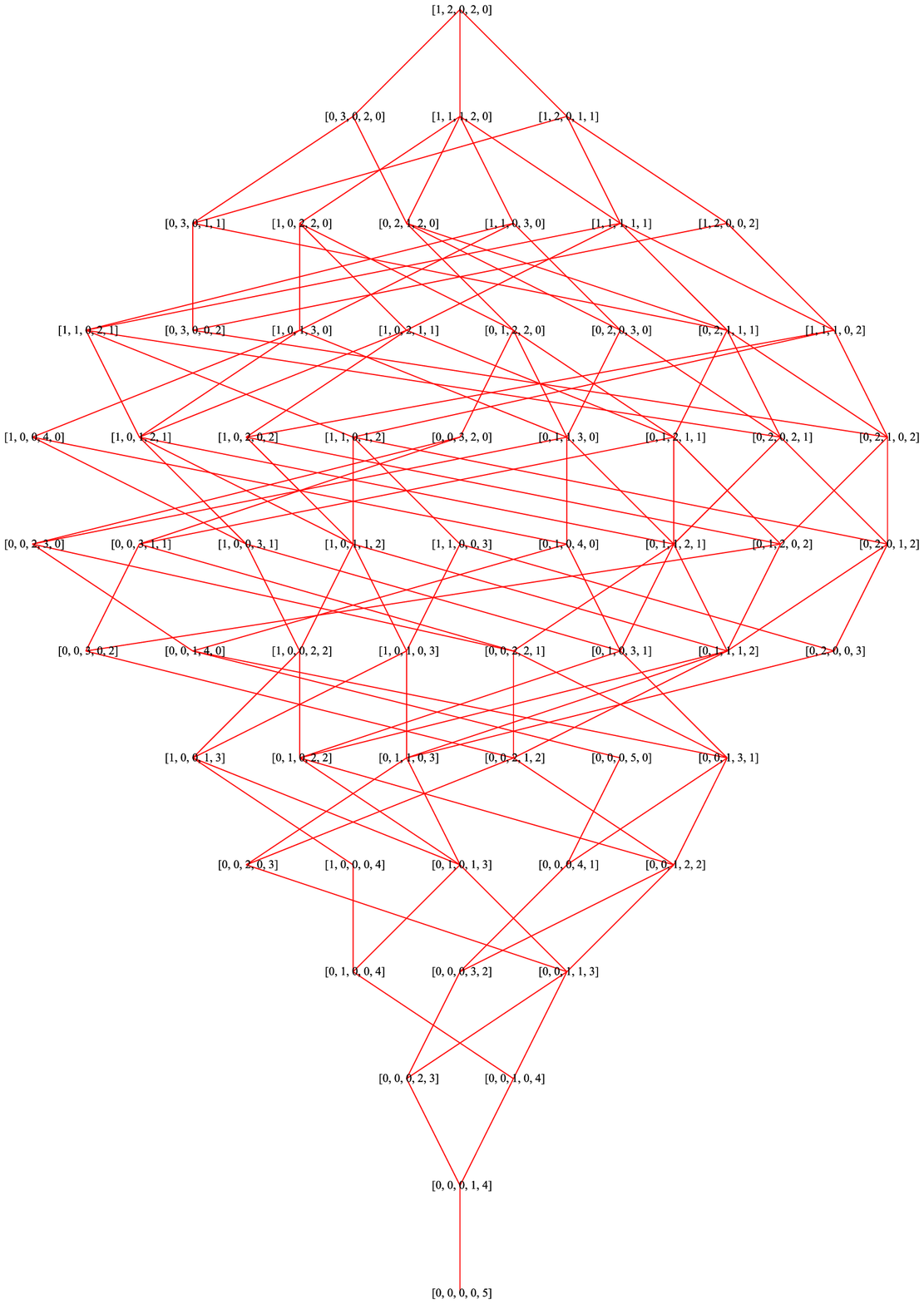}
	\caption{Poset structure of family SS7 }
	\label{fig:SS7b}
\end{figure}

\newpage

%%%%%%%%%%%%%%%%%%%%%%%%%%%%%%%%%%%%%%%%%%%%%%%%%%%%%%%%%%%%%%%%%%%%%%
\bibliographystyle{alpha}
\bibliography{bibliography}

\newcommand{\etalchar}[1]{$^{#1}$}
\def\cprime{$'$}
\begin{thebibliography}{BCGMB09}

\bibitem[All03]{allcock}
Daniel Allcock.
\newblock The moduli space of cubic threefolds.
\newblock {\em J. Algebraic Geom.}, 12(2):201--223, 2003.

\bibitem[BCGMB09]{beltrametti}
Mauro~C. Beltrametti, Ettore Carletti, Dionisio Gallarati, and Giacomo
  Monti~Bragadin.
\newblock {\em Lectures on curves, surfaces and projective varieties}.
\newblock EMS Textbooks in Mathematics. European Mathematical Society (EMS),
  Z\"urich, 2009.
\newblock A classical view of algebraic geometry, Translated from the 2003
  Italian original by Francis Sullivan.

\bibitem[CHSW85]{wittenstrominger}
P.~Candelas, Gary~T. Horowitz, Andrew Strominger, and Edward Witten.
\newblock Vacuum configurations for superstrings.
\newblock {\em Nuclear Phys. B}, 258(1):46--74, 1985.

\bibitem[Laz09]{laza}
Radu Laza.
\newblock The moduli space of cubic fourfolds.
\newblock {\em J. Algebraic Geom.}, 18(3):511--545, 2009.

\bibitem[LR08]{standardmonomial}
Venkatramani Lakshmibai and Komaranapuram~N. Raghavan.
\newblock {\em Standard monomial theory}, volume 137 of {\em Encyclopaedia of
  Mathematical Sciences}.
\newblock Springer-Verlag, Berlin, 2008.
\newblock Invariant theoretic approach, Invariant Theory and Algebraic
  Transformation Groups, 8.

\bibitem[Lun75]{luna}
D.~Luna.
\newblock Adh\'erences d'orbite et invariants.
\newblock {\em Invent. Math.}, 29(3):231--238, 1975.

\bibitem[MFK94]{mumford}
D.~Mumford, J.~Fogarty, and F.~Kirwan.
\newblock {\em Geometric invariant theory}, volume~34 of {\em Ergebnisse der
  Mathematik und ihrer Grenzgebiete (2) [Results in Mathematics and Related
  Areas (2)]}.
\newblock Springer-Verlag, Berlin, third edition, 1994.

\bibitem[MGH{\etalchar{+}}05]{maple}
Michael~B. Monagan, Keith~O. Geddes, K.~Michael Heal, George Labahn, Stefan~M.
  Vorkoetter, James McCarron, and Paul DeMarco.
\newblock {\em Maple~10 Programming Guide}.
\newblock Maplesoft, Waterloo ON, Canada, 2005.

\bibitem[Muk03]{mukai}
Shigeru Mukai.
\newblock {\em An introduction to invariants and moduli}, volume~81 of {\em
  Cambridge Studies in Advanced Mathematics}.
\newblock Cambridge University Press, Cambridge, 2003.
\newblock Translated from the 1998 and 2000 Japanese editions by W. M. Oxbury.

\bibitem[Sha80]{shahdeg2}
Jayant Shah.
\newblock A complete moduli space for {$K3$} surfaces of degree {$2$}.
\newblock {\em Ann. of Math. (2)}, 112(3):485--510, 1980.

\bibitem[Sha81]{shahdeg4}
Jayant Shah.
\newblock Degenerations of {$K3$} surfaces of degree {$4$}.
\newblock {\em Trans. Amer. Math. Soc.}, 263(2):271--308, 1981.

\bibitem[Ste98]{stembridgemaple}
John~R. Stembridge.
\newblock A maple package for posets - version 2.1, 1998.

\bibitem[VP89]{vinberg}
{\`E}.~B. Vinberg and V.~L. Popov.
\newblock Invariant theory.
\newblock In {\em Algebraic geometry, 4 ({R}ussian)}, Itogi Nauki i Tekhniki,
  pages 137--314, 315. Akad. Nauk SSSR Vsesoyuz. Inst. Nauchn. i Tekhn.
  Inform., Moscow, 1989.

\bibitem[Yok02]{yokoyama}
Mutsumi Yokoyama.
\newblock Stability of cubic 3-folds.
\newblock {\em Tokyo J. Math.}, 25(1):85--105, 2002.

\end{thebibliography}
%%%%%%%%%%%%%%%%%%%%%%%%%%%%%%%%%%%%%%%%%%%%%%%%%%%%%%%%%%%%%%%%%%%%%%

\end{document}